\documentclass[a4paper,11pt]{article}
\usepackage{amsopn, amssymb, amscd, amsmath, amsthm,amsfonts}
\usepackage{graphicx, graphics, epsfig}
\usepackage{color}

\newcommand{\nc}{\newcommand}

\nc{\RS}{\operatorname{RicSign}} \nc{\REV}{\operatorname{RicEV}}

\nc{\Gl}{\mathsf{GL}} \nc{\Or}{\mathsf{O}}  \nc{\SO}{\mathsf{SO}}  \nc{\SU}{\mathsf{SU}}  \nc{\Sl}{\mathsf{SL}} \nc{\Sp}{\mathsf{Sp}}
\nc{\G}{\mathsf{G}} \nc{\K}{\mathsf{K}}  \nc{\T}{\mathsf{T}} \nc{\Lsf}{\mathsf{L}}
\nc{\Qb}{\mathsf{Q}_\Beta} \nc{\Hb}{\mathsf{H}_\Beta} \nc{\Ub}{\mathsf{U}_\Beta}
\nc{\Gb}{\mathsf{G}_\Beta} \nc{\Kb}{\mathsf{K}_\Beta}
\nc{\PPP}{\mathsf{P}} \nc{\U}{\mathsf{U}} \nc{\N}{\mathsf{N}} \nc{\Ss}{\mathsf{S}} \nc{\Aa}{\mathsf{A}}

\nc{\laH}{\la\!\la} \nc{\raH}{\ra\!\ra}

\nc{\ipH}{{\laH \cdot, \cdot \raH}}

\nc{\Vg}{{V(\ggo)}}

\nc{\alert}{\color{blue}}

\nc{\fg}{\mathfrak{f}}  \nc{\vg}{\mathfrak{v}} \nc{\wg}{\mathfrak{w}} \nc{\zg}{\mathfrak{z}} \nc{\ngo}{\mathfrak{n}} \nc{\kg}{\mathfrak{k}} \nc{\mg}{\mathfrak{m}} \nc{\bg}{\mathfrak{b}} \nc{\ggo}{\mathfrak{g}} \nc{\ggob}{\overline{\mathfrak{g}}} \nc{\sog}{\mathfrak{so}} \nc{\sug}{\mathfrak{su}} \nc{\spg}{\mathfrak{sp}} \nc{\slg}{\mathfrak{sl}} \nc{\glg}{\mathfrak{gl}} \nc{\cg}{\mathfrak{c}} \nc{\rg}{\mathfrak{r}}  \nc{\hg}{\mathfrak{h}} \nc{\tgo}{\mathfrak{t}} \nc{\ug}{\mathfrak{u}} \nc{\dg}{\mathfrak{d}} \nc{\ag}{\mathfrak{a}} \nc{\pg}{\mathfrak{p}} \nc{\sg}{\mathfrak{s}} \nc{\affg}{\mathfrak{aff}} \nc{\qg}{\mathfrak{q}}
\nc{\Xg}{\mathfrak{X}} \nc{\lgo}{\mathfrak{l}}

\nc{\pca}{\mathcal{P}} \nc{\nca}{\mathcal{N}} \nc{\lca}{\mathcal{L}} \nc{\oca}{\mathcal{O}} \nc{\mca}{\mathcal{M}} \nc{\tca}{\mathcal{T}} \nc{\aca}{\mathcal{A}} \nc{\cca}{\mathcal{C}} \nc{\gca}{\mathcal{G}} \nc{\sca}{\mathcal{S}} \nc{\hca}{\mathcal{H}} \nc{\bca}{\mathcal{B}} \nc{\dca}{\mathcal{D}}

\nc{\vp}{\varphi} \nc{\ddt}{\tfrac{{\rm d}}{{\rm d}t}} \nc{\dds}{\tfrac{{\rm d}}{{\rm d}s}} \nc{\ddtbig}{\frac{{\rm d}}{{\rm d}t}} \nc{\dd}{{\rm d}}
\nc{\dpar}{\tfrac{\partial}{\partial t}} \nc{\im}{\mathtt{i}}

\nc{\RR}{{\mathbb R}} \nc{\HH}{{\mathbb H}} \nc{\CC}{{\mathbb C}} \nc{\ZZ}{{\mathbb Z}}
\nc{\FF}{{\mathbb F}} \nc{\NN}{{\mathbb N}} \nc{\QQ}{{\mathbb Q}} \nc{\PP}{{\mathbb P}}

\nc{\vs}{\vspace{.2cm}} \nc{\vsp}{\vspace{1cm}} \nc{\ip}{{\langle\cdot,\cdot\rangle}}
\nc{\ipp}{(\cdot,\cdot)} \nc{\la}{\langle} \nc{\ra}{\rangle} \nc{\unm}{\tfrac{1}{2}}
\nc{\unc}{\tfrac{1}{4}} \nc{\und}{\tfrac{1}{16}} \nc{\no}{\vs\noindent}
\nc{\lam}{\Lambda^2(\RR^n)^*\otimes\RR^n} \nc{\tangz}{{\rm T}^{\rm Zar}}
\nc{\lamg}{\Lambda^2\ggo^*\otimes\ggo}
\nc{\nor}{{\sf n}}  \nc{\mum}{/\!\!/} \nc{\kir}{/\!\!/\!\!/}
\nc{\Ri}{\tfrac{4\Ric_{\mu}}{||\mu||^2}} \nc{\ds}{\displaystyle}
\nc{\lb}{[\cdot,\cdot]} \nc{\isn}{\tfrac{1}{||v||^2}}
\nc{\gkp}{(\ggo=\kg\oplus\pg,\ip)} \nc{\ukh}{(\ug=\kg\oplus\hg,\ip)}
\nc{\tgkp}{(\tilde{\ggo}=\kg\oplus\pg,\ip)}
\nc{\wt}{\widetilde}
\nc{\raw}{\rightarrow} \nc{\lraw}{\longrightarrow} \nc{\hqn}{\mathcal{H}_{q,n}}

\nc{\Spec}{\operatorname{Spec}} \nc{\Nat}{\operatorname{nat}}

\nc{\ad}{\operatorname{ad}} \nc{\rk}{\operatorname{rk}} \nc{\Aut}{\operatorname{Aut}}   \nc{\Inn}{\operatorname{Inn}}   \nc{\Lie}{\operatorname{Lie}} \nc{\Ad}{\operatorname{Ad}} \nc{\Der}{\operatorname{Der}} \nc{\rad}{\operatorname{r}} \nc{\kf}{\operatorname{B}}

\nc{\End}{\operatorname{End}} \nc{\rank}{\operatorname{rank}} \nc{\Ker}{\operatorname{Ker}} \nc{\tr}{\operatorname{tr}} \nc{\sym}{\operatorname{sym}} \nc{\diag}{\operatorname{diag}} \nc{\proy}{\operatorname{pr}} \nc{\Adj}{\operatorname{Adj}} \nc{\vspan}{\operatorname{span}}

\nc{\Hess}{\operatorname{Hess}}  \nc{\dif}{\operatorname{d}} \nc{\sen}{\operatorname{sen}} \nc{\grad}{\operatorname{grad}} \nc{\Order}{\operatorname{O}} \nc{\divg}{\operatorname{div}}

\nc{\Iso}{\operatorname{Iso}} \nc{\Diff}{\operatorname{Diff}} \nc{\ricci}{\operatorname{ric}}  \nc{\Rc}{\operatorname{Rc}} \nc{\Ricci}{\operatorname{Ric}} \nc{\Riem}{\operatorname{Rm}} \nc{\scalar}{\operatorname{sc}} \nc{\scalarm}{\hat{\operatorname{R}}} \nc{\riccim}{\widehat{\operatorname{Ric}}} \nc{\tang}{\operatorname{T}} \nc{\vol}{\operatorname{vol}}

\nc{\mm}{\operatorname{M}} \nc{\CH}{\operatorname{CH}} \nc{\Irr}{\operatorname{Irr}} \nc{\mcc}{\operatorname{mcc}} \nc{\m}{\operatorname{m}} \nc{\pr}{\operatorname{pr}}

\nc{\Id}{\operatorname{Id}}  \nc{\mmm}{\operatorname{m}}

\numberwithin{equation}{section}

\theoremstyle{plain}
\newtheorem{theorem}{Theorem}[section]
\newtheorem{proposition}[theorem]{Proposition}
\newtheorem{corollary}[theorem]{Corollary}
\newtheorem{lemma}[theorem]{Lemma}

\theoremstyle{definition}

\theoremstyle{remark}
\newtheorem{remark}[theorem]{Remark}

\newtheorem{example}[theorem]{Example}

\author{Romina M. Arroyo\thanks{FaMAF $\&$ CIEM, Universidad Nacional de C\'ordoba, C\'ordoba, Argentina}~\thanks{School of Mathematics and Physics, The University of Queensland, St~Lucia,~QLD, Australia}~\thanks{Research supported by the Australian Government through the Australian Research Council's Discovery Projects funding scheme (DP180102185).} \\
\small{\texttt{arroyo@famaf.unc.edu.ar}} \and Mark D. Gould\footnotemark[2]~\footnotemark[3] \\ \small{\texttt{m.gould1@uq.edu.au}} \and Artem Pulemotov\footnotemark[2]~\footnotemark[3] \\
\small{\texttt{a.pulemotov@uq.edu.au}}}

\title{The prescribed Ricci curvature problem for naturally reductive metrics on non-compact simple Lie groups}

\begin{document}

\maketitle

\begin{abstract}
We investigate the prescribed Ricci curvature problem in the class of left-invariant naturally reductive Riemannian metrics on a non-compact simple Lie group. We obtain a number of conditions for the solvability of the underlying equations and discuss several examples.
\end{abstract}

\section{Introduction}

The study of the prescribed Ricci curvature problem is an important part of modern geometry with ties to flows, relativity and other subjects. The first wave of interest in this problem came in the 1980s; see~\cite[Chapter~5]{Bss} and~\cite[Section~6.5]{TA98}. Particularly extensive contributions were made at that time by DeTurck and his collaborators. For a discussion of the subsequent advances, including the recent progress in the framework on homogeneous spaces, see the survey~\cite{BP19}.

Let $M$ be a smooth manifold. In its original interpretation, the prescribed Ricci curvature problem comes down to the equation
\begin{align}\label{PRC_no_c}
\Ricci (g) = T,
\end{align}
where the Riemannian metric $g$ on $M$ is the unknown and the (0,2)-tensor field $T$ is given. The paper~\cite{DeTurck} 
proved, for nondegenerate~$T$, the existence of $g$ satisfying this equation in a neighbourhood of a point on~$M$; see also~\cite{DeTGold,AP13,AP16b}. However, subsequent research into the solvability of~\eqref{PRC_no_c} on \emph{all} of $M$ revealed the need for a more nuanced interpretation of the prescribed Ricci curvature problem. Specifically, suppose $M$ is closed and $T$ is positive-definite. The results of~\cite{Hamilton84,DeTurck85,Del03} and other papers suggest that there exists at most one $c\in\mathbb R$ such that the equation 
\begin{align}\label{PRC_c}
\Ricci (g) = cT
\end{align}
can be solved for $g$ on all of~$M$. This is certainly the case if $M$ is the 2-dimensional sphere and $T$ is positive-definite; see~\cite{WW70,DeTurck85,Hamilton84}. Thus, on a closed manifold, one customarily interprets the prescribed Ricci curvature problem as the question of finding $g$ and $c$ such that~\eqref{PRC_c} holds. This paradigm was originally proposed by DeTurck and Hamilton in~\cite{Hamilton84,DeTurck85}. As it turns out,~\eqref{PRC_c} arises in applications, such as the construction of the Ricci iteration; see~\cite{PR19,BPRZ19} and also~\cite[Section~3.10]{BP19}. On the other hand, if $M$ is open, it may be possible to obtain compelling existence theorems for~\eqref{PRC_no_c} without the additional constant~$c$. We refer to~\cite{Dela02,Dela18} for examples of such theorems.

In recent years, the third-named author and his collaborators produced a series of results, surveyed in~\cite{BP19}, on the prescribed Ricci curvature problem in the class of homogeneous metrics. More precisely, suppose $G$ is a connected Lie group. Let $M$ be a homogeneous space with respect to~$G$. Assume that the metric $g$ and the tensor field $T$ are $G$-invariant. Then \eqref{PRC_no_c} reduces to an overdetermined system of algebraic equations, whereas~\eqref{PRC_c} reduces to a determined one. For compact $M$ and positive-semidefinite $T$, the third-named author showed in~\cite{AP16} that homogeneous metrics satisfying~\eqref{PRC_c} are, up to scaling, critical points of the scalar curvature functional subject to the constraint~$\tr_gT=1$. This observation led to the discovery of several sufficient conditions for the solvability of~\eqref{PRC_c} in~\cite{AP16,MGAP18,AP20}. It parallels the well-known variational approach to the Einstein equation; see, e.g.,~\cite[\S1]{WZ86}. In the case where $M$ is compact but $T$ is not positive-semidefinite, the question of solvability of~\eqref{PRC_c} remains largely open. We hope that the present paper will stimulate its investigation; see Remark~\ref{rem_2c_Tim}.

As for the prescribed Ricci curvature problem for homogeneous metrics on \emph{non-compact} spaces, progress has been scarce so far. Buttsworth conducted in~\cite{TB19} a comprehensive study of this problem on unimodular three-dimensional Lie groups. In most of the cases he considered, there is at most one constant $c\in\mathbb R$ such that a metric $g$ satisfying~\eqref{PRC_c} exists. Several questions related to, but distinct from, the solvability of~\eqref{PRC_no_c} and~\eqref{PRC_c} on non-compact Lie groups have been studied by Milnor, Kowalski--Nikcevic, Eberlein, Kremlev--Nikonorov, Ha--Lee, Pina--dos Santos, the first-named author in collaboration with Lafuente (forthcoming work), and others. For a discussion of those results and a collection of references, see~\cite[Sections~2 and~4.1]{BP19}.

Left-invariant naturally reductive metrics on a Lie group form an important family nested between the set of all left-invariant metrics and the set of bi-invariant ones. The investigation of this family has led to several significant advances in geometry. For instance, it yielded new solutions to the Einstein equation and new insights into the spectral theory of the Laplacian on manifolds; see~\cite{DZ79,GS10,L19}. In the recent paper~\cite{APZ20}, Ziller and two of the authors studied~\eqref{PRC_c} for naturally reductive metrics on a compact Lie group using variational methods. That work exposed several interesting and previously unseen patterns of behaviour of the scalar curvature functional. For instance, one of the main theorems of~\cite{APZ20} is a necessary condition for the existence of a critical point subject to the constraint~$\tr_gT=1$. No results of this kind had appeared in the literature before.

The present paper studies the prescribed Ricci curvature problem for naturally reductive metrics on a \emph{non-compact} Lie group~$G$. We assume that $G$ is simple. The more general case of semisimple $G$ seems to be much more difficult analytically---we intend to consider it elsewhere. Since $G$ is an open manifold, it is reasonable for us to view the prescribed Ricci curvature problem as the question of finding solutions to~\eqref{PRC_no_c}. On the other hand, the fact that naturally reductive metrics are homogeneous suggests that a ``better" interpretation of this problem may be given by~\eqref{PRC_c}. The present paper studies both equations. We show that~\eqref{PRC_no_c} reduces to an overdetermined algebraic system. Even so, we are able to obtain a comprehensive existence and uniqueness theorem. Equation~\eqref{PRC_c} reduces to a determined system. In order to find conditions for solvability, we characterise metrics satisfying~\eqref{PRC_c} as critical points of the scalar curvature functional subject to one of three $T$-dependent constraints. While this characterisation is similar in spirit to the one obtained for compact Lie groups in~\cite{APZ20}, it bears some conceptual distinctions and requires a different proof. We obtain existence theorems for global maxima and classify some of the other critical points. The development and application of our variational methods presents many interesting analytical challenges and provides a wealth of insight for the investigation of~\eqref{PRC_c} on compact homogeneous spaces in the case of mixed-signature~$T$ (see, e.g.,~Remark~\ref{rem_2c_Tim}).

The paper is organised as follows. In Section~\ref{secnat}, we recall the characterisation, originally obtained by Gordon, of naturally reductive metrics on a non-compact simple Lie group. This characterisation underpins all of our results. In Section~\ref{secRicci}, we compute the Ricci curvature of a naturally reductive metric on~$G$. We believe that the formulas we obtain are of independent interest. Section~\ref{sec_Ricci=T} is devoted to equation~\eqref{PRC_no_c}. We produce a necessary and sufficient condition for the existence of a solution. We also establish uniqueness up to scaling. Section~\ref{sec_Ricci=cT} focuses on~\eqref{PRC_c}. We develop the variational approach to this equation and describe several types of critical points of the scalar curvature functional. At the end, we summarise the implications for the existence and the number of solutions. Section~\ref{sec_simple} examines the case where the metrics we consider are naturally reductive with respect to $G\times K$ for a simple subgroup~$K<G$. Here, we find conditions for the solvability of~\eqref{PRC_c} that are both necessary and sufficient. We also determine the precise number of solutions. Finally, Section~\ref{sec_examples} offers a series of examples.

\section{Naturally reductive metrics on non-compact simple Lie groups}\label{secnat}

Consider a connected non-compact simple Lie group $G$ with Lie algebra~$\ggo$. The results of~\cite[Section~5]{C85} yield a convenient characterisation of left-invariant naturally reductive metrics on~$G$. We present this characterisation in Theorem~\ref{G} below. For the basic theory of naturally reductive metrics, see~\cite[Section~1]{DZ79} and~\cite[Section~2]{C85}. In what follows, we identify every left-invariant (0,2)-tensor field on $G$ with the bilinear form it induces on~$\ggo$.

Let $K$ be a maximal compact subgroup of $G$ with Lie algebra~$\kg$. Suppose $B$ is the Killing form of~$\ggo$. Denote by $\pg$ the $B$-orthogonal complement of $\kg$ in~$\ggo$. Then
\begin{align*}
\ggo=\pg\oplus\kg
\end{align*}
is a Cartan decomposition. We have the inclusions
\begin{align*}
[\kg,\kg]\subset\kg,\qquad [\kg,\pg]\subset\pg,\qquad [\pg,\pg]\subset\kg.
\end{align*}
The Killing form $B$ is positive-definite on~$\pg$ and negative-definite on~$\kg$. Thus,
\begin{align*}
Q =B|_{\pg}-B|_{\kg}
\end{align*}
is an inner product on $\ggo$. Clearly, $Q$ is $\ad(\mathfrak k)$-invariant, and
\begin{equation}\label{sym_spce_incl}
Q([X,Y],Z)=-Q(X,[Y,Z]),\qquad X,Y \in \pg,~Z\in\kg.
\end{equation}

The quotient $G/K$ is a symmetric space. Because $G$ is simple, this space is irreducible. Consequently, the pair $(\ggo,\kg)$ must appear in Table~3 or~4 of~\cite[Section~7.H]{Bss}. Let $\kg_1,\ldots,\kg_r$ be the simple ideals of $[\kg,\kg]$. Denote by $\kg_{r+1}$ the centre of~$\kg$. Then
\begin{align}\label{dec_k}
\kg= \kg_1 \oplus \cdots \oplus \kg_{r+s},
\end{align}
where $s=0$ if $\kg_{r+1}$ is trivial and $s=1$ otherwise. Analysing the tables in~\cite[Section~7.H]{Bss}, we conclude that $\kg_{r+1}$ is at most 1-dimensional.

The direct product $G\times K$ acts on $G$ in accordance with the formula
\begin{align*}
(x,k)\,y=xyk^{-1},\qquad x,y\in G,~k\in K.
\end{align*}
The isotropy subgroup at the identity element of $G$ is 
\begin{align*}
\{(k,k)\in G\times K\,|\,k\in K\}.
\end{align*}
Denote by $\mathcal M_K$ the set of left-invariant metrics on $G$ that are naturally reductive with respect to~$G\times K$ and some decomposition of the Lie algebra of~$G\times K$. The main purpose of this paper is to study the prescribed Ricci curvature problem in~$\mathcal M_K$. Gordon showed in~\cite[Section~5]{C85} that every left-invariant naturally reductive metric on $G$ must lie in $\mathcal M_K$ for some choice of~$K$. Moreover, she obtained the following characterisation result.

\begin{theorem}[Gordon]\label{G}
A left-invariant metric $g$ on the simple group $G$ lies in $\mathcal M_K$ if and only if
\begin{equation}\label{metric}
g=\beta Q|_{\pg}+\alpha_1 Q|_{\kg_1}+\cdots+\alpha_{r+s}Q|_{\kg_{r+s}}
\end{equation}
for some $\beta,\alpha_1,\ldots,\alpha_{r+s}>0$.
\end{theorem}

\begin{remark} 
In~\cite{C85}, Gordon studied naturally reductive metrics on non-compact homogeneous spaces, not just on Lie groups. She obtained a version of Theorem~\ref{G} in this more general framework.
\end{remark}

\section{The Ricci curvature}\label{secRicci}

Our main objective in this section is to produce formulas for the Ricci curvature and the scalar curvature of a metric $g$ given by~(\ref{metric}). To do so, we need to introduce an array of constants, $\kappa_1,\ldots,\kappa_{r+s}$, associated with the pair~$(\ggo,\kg)$. Throughout the paper,
\begin{align*}
n=\dim\pg, \qquad d_i=\dim\kg_i, \qquad i=1,\ldots,r+s.
\end{align*}
As we explained above, $d_{r+1}=1$ if the centre of $\kg$ is non-trivial.

Suppose $B_i$ is the Killing form of~$\kg_i$ for $i=1,\ldots,r+s$. There exists $\kappa_i\in\mathbb R$ such that
\begin{align*}
B_i = \kappa_i B|_{\kg_i}.
\end{align*}
Using the assumption that $G$ is simple, one can easily check that $0<\kappa_i<1$ for $i=1,\ldots,r$. If the centre of $\kg$ is non-trivial, then $\kappa_{r+1}=0$.

Next, we state a proposition that provides a way of computing $\kappa_i$ for a specific pair~$(\ggo,\kg)$ and $i=1,\ldots,r$. In what follows, superscript~$\mathbb C$ means complexification. Clearly, the algebra $\ggo^{\mathbb C}$ is semisimple. We preserve the notation $\ad$ for the adjoint representation of~$\ggo^{\mathbb C}$. Choose Cartan subalgebras $H$ in $\ggo^{\mathbb C}$ and $H_i$ in~$\kg_i^{\mathbb C}$. It is easy to verify that $\ggo^{\mathbb C}$ is a completely reducible $\kg_i^{\mathbb C}$-module under the action given by~$\ad$. This observation implies that every element of $H_i$ must be semisimple in~$\ggo^{\mathbb C}$. Consequently, we may assume that $H$ contains~$H_i$. Let $\Phi^+$ and $\Phi_i^+$ be sets of positive roots of $\ggo^{\mathbb C}$ and~$\kg_i^{\mathbb C}$. The notation $\tr$ stands for the trace of a linear operator.

\begin{proposition}\label{prop_kappa}
Given $i=1,\ldots,r$ and $Z\in H_i$, the constant $\kappa_i$ satisfies
\begin{align*}
\sum_{\nu\in\Phi_i^+}(\nu(Z))^2=\kappa_i\sum_{\nu\in\Phi^+}(\nu(Z))^2.
\end{align*}
\end{proposition}

\begin{proof}
We preserve the notation $B$ and $B_i$ for the Killing forms of $\ggo^{\mathbb C}$ and $\kg_i^{\mathbb C}$. Because $K$ is compact, $\kg_i^{\mathbb C}$ is a simple subalgebra of~$\ggo^{\mathbb C}$. Therefore,
\begin{align*}
B_i(Z,Z)=\kappa_iB(Z,Z).
\end{align*}
Using basic properties of root systems, we find
\begin{align*}
B_i(Z,Z)&=\tr\!\big(\!\ad(Z)\ad(Z)|_{\kg_i^{\mathbb C}}\big)=2\sum_{\nu\in\Phi_i^+}(\nu(Z))^2, \\ B(Z,Z)&=\tr(\ad(Z)\ad(Z))=2\sum_{\nu\in\Phi^+}(\nu(Z))^2.
\end{align*}
\end{proof}

\begin{remark}
The assertion of Proposition~\ref{prop_kappa} holds even if $G$ is not simple but merely semisimple.
\end{remark}

\begin{remark}
One can use properties of Casimir elements to produce another formula for~$\kappa_i$. More precisely, given $i=1,\ldots,r$, there exists a decomposition
\begin{align*}
\pg^{\mathbb C}=\pg_1^i\oplus\cdots\oplus\pg_{r_i}^i
\end{align*}
such that every $\pg_j^i$ is a non-trivial irreducible $\kg_i^{\mathbb C}$-module under the action defined by~$\ad$. Let $\psi_j^i$ be the highest weight of~$\pg_j^i$. Denote by $\rho_i$ the half-sum of positive roots of~$\kg_i^{\mathbb C}$. Using classical results on eigenvalues of Casimir elements, one can show that
\begin{align*}
\kappa_i=\frac{d_i}{d_i+\sum_{j=1}^{r_i}B_i(\psi_j^i,\psi_j^i+2\rho_i)\dim\pg_j^i},
\end{align*}
where we preserve the notation $B_i$ for the bilinear form induced on $H_i^*$ by the Killing form of~$\kg_i^{\mathbb C}$. Related formulas can be found in Dynkin's work; see~\cite{Dyn57}.
\end{remark}

\begin{example}\label{exa_kappa}
Assume $G=\SU(p,q)$ and $K=\SU(p)\times U(q)$ with $2\le p\le q$. Then $r=2$ and $s=1$ in the decomposition~\eqref{dec_k}. Clearly,
\begin{align*}
\ggo^{\mathbb C}=\slg(p+q,\mathbb C),\qquad \kg_1^{\mathbb C}=\slg(p,\mathbb C),\qquad \kg_2^{\mathbb C}=\slg(q,\mathbb C),\qquad \kg_3^{\mathbb C}=\mathbb C.
\end{align*}
Denote by $E_i^j$ the matrix of size $(p+q)\times(p+q)$ that has 1 in the $(i,j)$th slot and 0 elsewhere. Suppose
\begin{align*}
&H=\bigg\{\sum_{i=1}^{p+q-1}\lambda_i\big(E_i^i-E_{i+1}^{i+1}\big)\,\bigg|\,\lambda_i\in\mathbb C\bigg\}, & &H_1=\bigg\{\sum_{i=1}^{p-1}\lambda_i\big(E_i^i-E_{i+1}^{i+1}\big)\,\bigg|\,\lambda_i\in\mathbb C\bigg\}, \\
&\Phi^+=\{\epsilon_i-\epsilon_j\,|\,1\le i<j\le p+q\},& &\Phi_1^+=\{\epsilon_i-\epsilon_j\,|\,1\le i<j\le p\},
\end{align*}
where $\epsilon_i$ is the linear functional on $H$ such that $\epsilon_i\big(E_j^j-E_{j+1}^{j+1}\big)$ is the difference of Kronecker deltas~$\delta_i^j-\delta_i^{j+1}$. Choosing $Z=E_1^1-E_2^2$, we find
\begin{align*}
\sum_{\nu\in\Phi^+}(\nu(Z))^2=&\big((\epsilon_1-\epsilon_2)\big(E_1^1-E_2^2\big)\big)^2+\sum_{i=3}^{p+q}\big((\epsilon_1-\epsilon_i)\big(E_1^1-E_2^2\big)\big)^2 \\
&+\sum_{i=3}^{p+q}\big((\epsilon_2-\epsilon_i)\big(E_1^1-E_2^2\big)\big)^2=2(p+q),\\
\sum_{\nu\in\Phi_1^+}(\nu(Z))^2=&2p.
\end{align*}
Proposition~\ref{prop_kappa} implies $\kappa_1=\frac p{p+q}$. A similar argument with $Z=E_{p+1}^{p+1}-E_{p+2}^{p+2}$ yields $\kappa_2=\frac q{p+q}$. Since $\kg_3$ is abelian, $\kappa_3=0$.
\end{example}

If $r=1$, one can calculate $\kappa_i$ using formula~\eqref{sum_dkappan} below; see Examples~\ref{exa_simple} and~\ref{ex_2ideals}.

\begin{remark}
The work~\cite{DZ79} computes a range of constants analogous to $\kappa_i$ in the framework of compact Lie groups. One can find $\kappa_i$ for a specific pair $(\ggo,\kg)$ using those results along with duality for symmetric spaces; see, e.g.,~\cite[Sections~7.82--7.83]{Bss}. In the present paper, we choose a more direct approach.
\end{remark}

We are now ready to state the main result of this section.

\begin{theorem}\label{Ric} 
Suppose $g\in\mca_K$ is a naturally reductive metric on the simple group~$G$ satisfying~(\ref{metric}). The Ricci curvature of $g$ is given by the formulas
	\begin{align*}
	&\Ricci (g)|_{\pg} = - \sum_{i=1}^{r+s} \Big(\frac{\alpha_i}{2 \beta}+1\Big)\frac{ d_i(1-\kappa_i)}{n}  Q|_{\pg}, \\ 
	&\Ricci (g)|_{\kg_j} = \unc\Big(\frac{\alpha_j^2}{\beta^2}(1-\kappa_j)+\kappa_j\Big)Q|_{\kg_j}, \\
	&\Ricci (g)(\pg,\kg_j)=\Ricci (g)(\kg_j,\kg_k)=0, &j,k =1,\ldots,r+s,~j\ne k.
	\end{align*}
\end{theorem}

To prove Theorem~\ref{Ric}, we apply the strategy developed in~\cite[Section~5]{DZ79}. In what follows, $\tr_h$ stands for the trace of a bilinear form with respect to an inner product~$h$. The notation $\pi_\ug$ is used for the $Q$-orthogonal projection onto $\ug\subset\ggo$. For $j=1,\ldots,r+s$, define a bilinear form $A_j$ on $\pg$ by setting
\begin{align*}
A_j(X,Y)=\tr(\pi_{\kg_j}\ad(X)\ad(Y)).
\end{align*}
Fix $Q$-orthonormal bases $(v_i)_{i=1}^n$ of $\pg$ and $\big(v_k^j\big)_{k=1}^{d_j}$ of $\kg_j$. We need the following auxiliary result; cf.~\cite[pages~609--610]{Jen73} and~\cite[pages~32--34]{DZ79}.

\begin{lemma}\label{lemma_Ai}
Given $j=1,\ldots,r+s$,
\begin{align*}
\tr_{Q|_\pg}A_j=d_j(1-\kappa_j),\qquad \sum_{i=1}^{r+s}A_i=\unm Q|_{\pg}.
\end{align*}
\end{lemma}

\begin{proof}
Invoking~\eqref{sym_spce_incl}, we compute
\begin{align*}
\tr_{Q|_\pg}A_j&=\sum_{i=1}^n\sum_{k=1}^{d_j}Q\big([v_i,[v_i,v_k^j]],v_k^j\big)
=-\sum_{i=1}^n\sum_{k=1}^{d_j}Q\big(v_i,[[v_i,v_k^j],v_k^j]\big)
\\ &=-\tr_{Q|_{\kg_j}}B|_{\kg_j}+\sum_{k=1}^{d_j}\sum_{l=1}^{r+s}\sum_{m=1}^{d_l}Q\big([v_k^j,[v_k^j,v_m^l]],v_m^l\big).
\end{align*}
Since $\kg_j$ is an ideal of $\kg$,
\begin{align*}
\sum_{l=1}^{r+s}\sum_{m=1}^{d_l}\sum_{k=1}^{d_j}Q\big([v_k^j,[v_k^j,v_m^l]],v_m^l\big)&=
\sum_{m=1}^{d_j}\sum_{k=1}^{d_j}Q\big([v_k^j,[v_k^j,v_m^j]],v_m^j\big) =\tr_{Q|_{\kg_j}}B_j=\kappa_j\tr_{Q|_{\kg_j}}B|_{\kg_j}.
\end{align*}
We conclude that
\begin{align*}
\tr_{Q|_\pg}A_j=-\tr_{Q|_{\kg_j}}B|_{\kg_j}+\kappa_j\tr_{Q|_{\kg_j}}B|_{\kg_j}=d_j(1-\kappa_j).
\end{align*}
This proves the first equality in the statement of the lemma.

For $X,Y\in\pg$,
\begin{align*}
\sum_{i=1}^{r+s}A_i(X,Y)&=
\sum_{i=1}^{r+s}\sum_{k=1}^{d_i}Q\big([X,[Y,v_k^i]],v_k^i\big)
=B(X,Y)-\sum_{l=1}^nQ([X,[Y,v_l]],v_l).
\end{align*}
It is easy to see that the forms $A_i$ are symmetric. Consequently,
\begin{align*}
\sum_{l=1}^nQ([X,[Y,v_l]],v_l)=\tr(\ad(X)\pi_\kg\ad(Y))=\tr(\pi_\kg\ad(Y)\ad(X))=\sum_{i=1}^{r+s}A_i(X,Y).
\end{align*}
We conclude that
\begin{align*}
\sum_{i=1}^{r+s}A_i(X,Y)=B(X,Y)-\sum_{i=1}^{r+s}A_i(X,Y),
\end{align*}
which implies the second inequality.
\end{proof}

\begin{proof}[Proof of Theorem~\ref{Ric}]
Let $\nabla$ be the Levi-Civita connection of the metric~$g$. The Koszul formula yields
\begin{align}\label{Koszul}
\nabla_X Y=\begin{cases}
\unm [X,Y] & \mbox{if}~X,Y \in \pg~\mbox{or}~X,Y\in\kg, \\
-\frac{\alpha_i}{2 \beta} [X,Y] & \mbox{if}~X \in \pg~\mbox{and}~Y \in \kg_i~\mbox{for some}~i=1, \ldots, r+s, \\
\big(\frac{\alpha_i}{2 \beta}+1\big)[X,Y] & \mbox{if}~X \in \kg_i~\mbox{for some}~i=1, \ldots, r+s~\mbox{and}~Y \in \pg;
\end{cases}
\end{align}
see~\cite[page~485]{C85}. The Ricci curvature of $g$ satisfies
\begin{align}\label{Ric_AA}
\Ricci(g)(X,Y)
=-\tr\nabla_{\nabla_\cdot Y}X
,\qquad X,Y\in\ggo.
\end{align}
This fact goes back to~\cite{Sag70}; a simpler proof appeared in~\cite[Section~5]{DZ79}. Substituting~\eqref{Koszul} into~\eqref{Ric_AA}, we easily obtain the required identities for $\Ricci (g)|_{\kg_j}$, $\Ricci (g)(\pg,\kg_j)$ and $\Ricci (g)(\kg_j,\kg_k)$ with $j\ne k$; cf.~\cite[page~33]{DZ79}.

Because $g$ lies in $\mca_K$, it is $\ad(\kg)$-invariant. Consequently, there exists $\tau\in\mathbb R$ such that $\Ricci(g)|_{\pg}=\tau Q|_{\pg}$. Taking the trace with respect to $Q|_{\pg}$ on both sides and exploiting Lemma~\ref{lemma_Ai}, we find
\begin{align*}
n\tau&=\sum_{i=1}^n\Ricci(g)(v_i,v_i)=-\sum_{i=1}^n\bigg(\sum_{l=1}^nQ(\nabla_{\nabla_{v_l}v_i}v_i,v_l)+\sum_{j=1}^{r+s}\sum_{m=1}^{d_j}Q\big(\nabla_{\nabla_{v_m^j}v_i}v_i,v_m^j\big)\bigg) 
\\
&=-\unm\sum_{i=1}^n\sum_{j=1}^{r+s}\Big(\frac{\alpha_j}{2\beta}+1\Big)\sum_{m=1}^{d_j}\bigg(\sum_{l=1}^nQ\big([v_l,v_i],v_m^j\big)Q\big([v_m^j,v_i],v_l\big)+Q\big([[v_m^j,v_i],v_i],v_m^j\big)\bigg)
\\
%
%
&=-\sum_{i=1}^n\sum_{j=1}^{r+s}\Big(\frac{\alpha_j}{2\beta}+1\Big)\sum_{m=1}^{d_j}Q\big([v_i,[v_i,v_m^j]],v_m^j\big)
\\
&=-\sum_{j=1}^{r+s}\Big(\frac{\alpha_j}{2\beta}+1\Big)\tr_{Q|_{\pg}}A_j=-\sum_{j=1}^{r+s}\Big(\frac{\alpha_j}{2\beta}+1\Big)d_j(1-\kappa_j).
\end{align*}
Therefore,
\begin{align*}
\tau=-\sum_{j=1}^{r+s}\Big(\frac{\alpha_j}{2\beta}+1\Big)\frac{d_j(1-\kappa_j)}n,
\end{align*}
which yields the required identity for~$\Ricci(g)|_{\pg}$.
\end{proof}

Denote by $S$ the scalar curvature functional on~$\mca_K$. Our next goal is to produce a formula for~$S$. Taking the trace of the second equality in Lemma~\ref{lemma_Ai} and using the first one, we obtain
\begin{align}\label{sum_dkappan}
2\sum_{i=1}^{r+s} d_i (1-\kappa_i) = n.
\end{align}
Theorem~\ref{Ric} and~\eqref{sum_dkappan} imply the following result.

\begin{corollary}\label{scalar}
Suppose $g\in\mca_K$ is a naturally reductive metric on the simple group~$G$ satisfying~(\ref{metric}). The scalar curvature of $g$ is given by the formula
\begin{equation*}
S(g) = -\unc \sum_{i=1}^{r+s} \frac{\alpha_i}{\beta^2} d_i (1-\kappa_i) -  \frac{n}{2 \beta} + \unc \sum_{i=1}^{r} \frac{d_i\kappa_i}{\alpha_i}.
\end{equation*} 
\end{corollary}

\section{Metrics with prescribed Ricci curvature}\label{sec_Ricci=T}

Consider a (0,2)-tensor field $T$ on the simple group $G$. In this section, we state a necessary and sufficient condition for the solvability of the equation
\begin{align}\label{bdy_PRC_no_c}
\Ricci(g)=T
\end{align}
in the class $\mathcal M_K$. If a metric $g\in\mathcal M_K$ satisfying~\eqref{bdy_PRC_no_c} exists, then $T$ must be left-invariant. Moreover, by Theorem~\ref{Ric}, the formula
\begin{equation}\label{T}
T = -T_{\pg} Q|_{\pg} + T_1 Q|_{\kg_1} + \ldots + T_{r+s} Q|_{\kg_{r+s}}
\end{equation}
holds with $T_{\pg},T_1,\ldots,T_{r+s}>0$.

\begin{theorem}\label{thm_no_c}
Suppose $T$ is a left-invariant (0,2)-tensor field on $G$ given by~(\ref{T}). A naturally reductive metric $g\in\mathcal M_K$ satisfying~(\ref{bdy_PRC_no_c}) exists if and only if
\begin{align}\label{cond1_cno}
4T_i -\kappa_i>0
\end{align}
for all $i=1,\ldots,r$ and
\begin{align}\label{cond_noc}
T_{\pg}=\sum_{i=1}^{r+s}\frac{2d_i(1-\kappa_i)+d_i\sqrt{(4T_i -\kappa_i)(1-\kappa_i)}}{2n}.
\end{align}
There is at most one such $g$, up to scaling.
\end{theorem}

\begin{proof}
Consider a naturally reductive metric $g\in\mathcal M_K$. It satisfies~\eqref{metric} for some $\beta,\alpha_1,\ldots,\alpha_{r+s}>0$. By Theorem~\ref{Ric}, the Ricci curvature of $g$ equals $T$ if and only if
\begin{align*}
\frac{\alpha_i}{\beta}&=\sqrt{\frac{4T_i -\kappa_i}{1-\kappa_i}}, \qquad  1\leq i \leq r+s, \\ 
T_{\pg}&=\sum_{i=1}^{r+s} \Big(1+\frac{\alpha_i}{2 \beta}\Big)\frac{(1-\kappa_i) d_i}{n}
=\sum_{i=1}^{r+s}\frac{2d_i(1-\kappa_i)+d_i\sqrt{(4T_i -\kappa_i)(1-\kappa_i)}}{2n}.
\end{align*}
This observation proves the result.
\end{proof}

It is tempting to use Theorem~\ref{thm_no_c} to study the solvability of the equation
\begin{equation}\label{prescribed}
\Ricci (g)= cT
\end{equation}
in the class $\mca_K$. Indeed, suppose $\Xi$ is the set of left-invariant tensor fields on $G$ satisfying~\eqref{T}, \eqref{cond1_cno} and~\eqref{cond_noc}. Theorem~\ref{thm_no_c} states that~\eqref{bdy_PRC_no_c} has a solution if and only if $T$ lies in~$\Xi$. Using this result, we can easily obtain a description of the set of tensor fields that coincide with Ricci curvatures of metrics in $\mca_K$ up to scaling. Namely, a pair $(g,c)\in\mca_K\times(0,\infty)$ satisfying~\eqref{prescribed} exists if and only if
\begin{align*}
T\in\Xi'=\{\tau h\,|\,\tau>0~\mbox{and}~h\in\Xi\}.
\end{align*}
However, in general, it is difficult to determine whether a specific $T$ given by~\eqref{T} lies in~$\Xi'$. To do so, one has to check whether~\eqref{cond1_cno} and~\eqref{cond_noc} hold with $T_{\mathfrak p},T_1,\ldots,T_{r+s}$ replaced by $cT_{\mathfrak p},cT_1,\ldots,cT_{r+s}$ for some $c>0$. 
Already when $r+s=2$, this involves the tricky task of understanding if a polynomial of degree~$4$ has roots that 
obey several constraints; when $r+s\ge3$, the question seems to be substantially harder. In the present paper, we take a different approach to the analysis of~\eqref{prescribed}. We are able to show that the existence of a pair $(g,c)\in\mca_K\times(0,\infty)$ satisfying~\eqref{prescribed} follows from simple inequalities for the components of~$T$. Moreover, we draw interesting conclusions regarding the non-uniqueness of such a pair.

\section{Metrics with Ricci curvature prescribed up to scaling}\label{sec_Ricci=cT}

Suppose $T$ is a left-invariant symmetric (0,2)-tensor field on~$G$. Our next goal is to study the solvability of equation~\eqref{prescribed} in the class $\mathcal M_K$. As above, we assume the group $G$ is simple. This implies, in particular, that $\kappa_i<1$ for all~$i$.

First, we re-state the problem in variational terms. More precisely, define
\begin{align}\label{MKT_def}
\mathcal{M}^+_{T} &=\{g \in \mathcal{M}_ K\,|\,\tr_{g}T=1\},\notag\\
\mathcal{M}^-_{T} & =\{g \in \mathcal{M}_ K\,|\,\tr_{g}T=-1\},\notag \\
\mathcal{M}^0_{T} &=\{g \in \mathcal{M}_ K\,|\,\tr_{g}T=0\}.
\end{align}
Each of these three spaces carries a manifold structure induced from~$\mathcal M_K$. In Section~\ref{sec_variational}, we show that $g$ satisfies~\eqref{prescribed} if and only if it is (up to scaling) a critical point of $S|_{\mathcal{M}^+_{T}}$, $S|_{\mathcal{M}^-_{T}}$ or $S|_{\mathcal{M}^0_{T}}$. This resembles the variational interpretation of~\eqref{prescribed} on compact Lie groups for positive-semidefinite~$T$ (see~\cite[Proposition~3.1]{APZ20}); however, in that case, only $S|_{\mathcal{M}^+_{T}}$ needs to be considered. In Sections~\ref{sec_+} and~\ref{sec_-}, we obtain sufficient conditions for the existence of global maxima of $S|_{\mathcal{M}^+_{T}}$ and $S|_{\mathcal{M}^-_{T}}$, respectively. This requires complex estimates on the scalar curvature obtained in Lemmas~\ref{lemmacompact+} and~\ref{lemmacompact-}. In Section~\ref{sec_0}, we classify completely the critical points of~$S|_{\mathcal{M}^0_{T}}$ assuming $r+s\le2$. The analysis here can be reduced, as Lemma~\ref{lem_crit=root} demonstrates, to the study of a cubic polynomial in one variable. Finally, in Section~\ref{sec_sum}, we summarise the implications of our results for the prescribed Ricci curvature problem.

\begin{remark}\label{rem_2c_Tim}
It appears that~\eqref{prescribed} admits a similar variational interpretation on compact homogeneous spaces when $T$ has  mixed signature. Thus, our arguments yield new insight into the study of~\eqref{prescribed} in that setting. For instance, Buttsworth showed in~\cite{TB19} through methods of elementary polynomial analysis that, for certain $T$ on~$\SU(2)$, a left-invariant metric $g$ satisfying~\eqref{prescribed} exists for precisely two distinct constants~$c\in\mathbb R$. This was somewhat surprising at the time, as nothing similar had occurred in previously understood examples. We observe an analogous phenomenon in Theorem~\ref{thm_summary} below. In our arguments, the two constants arise naturally as Lagrange multipliers for $S|_{\mathcal{M}^+_{T}}$ and $S|_{\mathcal{M}^-_{T}}$.
\end{remark}

\subsection{The variational approach}\label{sec_variational}

The following result underpins our approach to the study of~\eqref{prescribed}.

\begin{proposition}\label{variational_lemma}
Let $T$ be a left-invariant symmetric (0,2)-tensor field on $G$. A metric $g \in\mathcal{M}_{K}$ satisfies~(\ref{prescribed}) for some $c \in \RR$ if and only if it is (up to scaling) a critical point of~$S |_{\mathcal{M}^{+}_{T}}$, $S |_{\mathcal{M}^{-}_{T}}$ or~$S |_{\mathcal{M}^{0}_{T}}$.
\end{proposition}

\begin{proof}
Denote by $\tca$ the space of left-invariant bilinear form fields
\begin{equation*}
\gamma Q|_{\pg} + \gamma_1 Q|_{\kg_1} + \ldots + \gamma_{r+s} Q|_{\kg_{r+s}}
\end{equation*}
with $\gamma,\gamma_1,\ldots,\gamma_{r+s}\in\mathbb R$. Theorem~\ref{Ric} shows that $\Ricci(g)$ lies in~$\tca$. We identify $\tca$ with the space tangent to~$\mathcal M_K$ at~$g$ in the natural way. The left-invariant bilinear form fields
\begin{align*}
Q_{\pg}=\pi_{\pg}^*Q,\qquad Q_i=\pi_{\kg_i}^*Q,\qquad i=1,\ldots,r+s,
\end{align*}
where $*$ denotes pullback, make a basis of~$\tca$.

Suppose $g$ is given by~\eqref{metric}. Using Corollary~\ref{scalar}, \eqref{sum_dkappan} and Theorem~\ref{Ric}, we find that the differential of the scalar curvature functional $S$ satisfies
\begin{align*}
d S_g (Q_{\pg})&=\frac{\partial}{\partial \beta} \bigg(-\unc \sum_{i=1}^{r+s} \frac{\alpha_i}{\beta^2} d_i (1-\kappa_i) -  \frac{n}{2 \beta}\bigg)
\\ 
&= \frac{1}{\beta^2}\sum_{i=1}^{r+s}\Big(\frac{\alpha_i}{2\beta}+1\Big)d_i (1-\kappa_i)=- g( \Ricci (g), Q_{\pg} ), 
\\ d S_g (Q_j) &= \frac{\partial}{\partial \alpha_j}\bigg( -\unc \sum_{i=1}^{r+s} \frac{\alpha_i}{\beta^2} d_i (1-\kappa_i)+\unc \sum_{i=1}^{r} \frac{d_i\kappa_i}{\alpha_i} \bigg) \\
&=-\frac{d_j}{4\alpha_j^2} \Big(\frac{\alpha_j^2}{4 \beta^2}(1-\kappa_j)+\kappa_j\Big)=- g( \Ricci (g), Q_j), \qquad j=1, \ldots, r+s.
\end{align*}
(We preserve the notation~$g$ for the inner product induced by $g$ on the tensor bundle over~$G$.) Consequently,
\begin{equation}\label{sc=Ric}
d S_g (h) =-g(\Ricci (g),h)
\end{equation}
for all $h \in \tca$.

Let us scale $g$ by the factor
\begin{align*}
\tau=
\begin{cases}
|\tr_gT| &\mbox{if}~\tr_gT\ne0, \\
1 &\mbox{if}~\tr_gT=0.
\end{cases}
\end{align*}
Clearly, $\tau g$ lies in~$\mathcal M_T^\sigma$ for some $\sigma\in\{+,-,0\}$. The space tangent to $\mathcal M^\sigma_T$ at $\tau g$ consists of those $h \in\tca $ that satisfy $g(T,h)=0$. Formula~\eqref{sc=Ric} implies that $d S_{\tau g}$ vanishes on this space if and only if $g$ satisfies~\eqref{prescribed}.
\end{proof}

By Theorem~\ref{Ric}, the constant $c$ in Proposition~\ref{variational_lemma} must be positive if $T$ is given by~\eqref{T} with $T_{\pg},T_1,\ldots,T_{r+s}>0$. The above proof shows that one may think of $c$ as a Lagrange multiplier.

\subsection{Global maxima on $\mca_{T}^{+}$}\label{sec_+}

Our goal in this subsection is to show that simple inequalities for $T$ guarantee the existence of a critical point of~$S|_{\mca_{T}^{+}}$.

\begin{theorem}\label{sufcon+}
Suppose $T$ is a left-invariant (0,2)-tensor field on $G$ satisfying~(\ref{T}) for some $T_{\pg},T_1,\ldots,T_{r+s}>0$. Choose an index $m$ such that 
\begin{align}\label{def_index_m}
\frac{\kappa_m}{T_m}=\max_{i=1,\ldots,r} \frac {\kappa_i}{T_i}.
\end{align}
If 
\begin{equation}\label{hyp_thm_beta_inf+}
\frac{\kappa_m\tr_Q T}{T_m}< \dim \kg + d_m (1-\kappa_m) - 3 n
\end{equation}
and
\begin{equation}\label{hyp_thm_beta_zero+}
\sum_{i=1}^{r+s} \frac{n^2\kappa_iT_{\mathfrak p}^2- d_i^2(1-\kappa_i)T_i^2}{nT_{\mathfrak p}T_i}-2n<0,
\end{equation}
then the functional $S |_{\mathcal{M}^{+}_{T}}$ attains its global maximum.
\end{theorem}

The proof of Theorem~\ref{sufcon+} requires the following estimate for $S |_{\mathcal{M}^{+}_{T}}$.

\begin{lemma}\label{lemmacompact+}
Let $m$ be as in~(\ref{def_index_m}). Assume that~(\ref{hyp_thm_beta_zero+}) holds. Given $\epsilon > 0$, there exists a compact set $\mathcal C_\epsilon^+\subset\mathcal M^+_T$ such that
\begin{equation}\label{estimate+}
S(g) < \frac{\kappa_m}{4T_m} + \epsilon
\end{equation}
for every $g \in \mathcal{M}^+_{T} \setminus\cca_\epsilon^+$.
\end{lemma}

\begin{proof}
Consider a metric $g\in\mathcal M^+_T$ satisfying~\eqref{metric}.  The definition of $\mathcal M^+_T$ implies
\begin{align}\label{trace=1}
\tr_gT=-\frac{nT_{\mathfrak p}}\beta+\sum_{i=1}^{r+s}\frac{d_iT_i}{\alpha_i}=1.
\end{align}
For $\epsilon>0$, denote
\begin{align*}
\Lambda_\infty^+(\epsilon)=\frac{n\kappa_mT_{\mathfrak p}}{4T_m\epsilon}.
\end{align*}
In view of Corollary~\ref{scalar} and formula~\eqref{trace=1}, if $\beta>\Lambda_\infty^+(\epsilon)$, then
\begin{equation*}
	S(g) < \unc \sum_{i=1}^{r} \frac{d_i\kappa_i}{\alpha_i}\le\frac{\kappa_m}{4T_m}\sum_{i=1}^{r+s} \frac{d_iT_i}{\alpha_i}=\frac{\kappa_m}{4T_m}\bigg(1+\frac{nT_{\mathfrak p}}\beta\bigg)<\frac{\kappa_m}{4T_m}+\epsilon.
	\end{equation*}
Thus, in this case, estimate~\eqref{estimate+} holds.

Denote
\begin{align*}
\Lambda_0^+=\unm\bigg(\sum_{i=1}^{r+s}\frac{n^2\kappa_iT_{\pg}^2+d_i^2(1-\kappa_i)T_i^2}{n^2T_{\mathfrak p}^2T_i}\bigg)^{-1}\bigg|\sum_{i=1}^{r+s} \frac{n^2\kappa_iT_{\mathfrak p}^2-d_i^2 (1-\kappa_i)T_i^2}{nT_{\mathfrak p}T_i}-2n\bigg|.
\end{align*}
Formula~\eqref{trace=1} implies
\begin{align*}
\frac{\beta}{\alpha_j}<\frac\beta{d_jT_j}\sum_{i=1}^{r+s}\frac{d_iT_i}{\alpha_i}=\frac{\beta+nT_{\mathfrak p}}{d_jT_j},\qquad j=1,\ldots,r+s.
\end{align*}
Invoking Corollary~\ref{scalar} again, we find
\begin{align*}
S(g) &=\frac1{4\beta}\bigg(\sum_{i=1}^{r+s} \Big(d_i\kappa_i\frac{\beta}{\alpha_i}-d_i (1-\kappa_i)\frac{\alpha_i}{\beta}\Big)-2n \bigg)
\\
& <\frac1{4\beta}\bigg(\sum_{i=1}^{r+s} \bigg(\kappa_i\frac{\beta+nT_{\mathfrak p}}{T_i}-\frac{d_i^2(1-\kappa_i)T_i}{\beta+nT_{\mathfrak p}}\bigg)-2n \bigg)
\\
& =\frac1{4\beta}\bigg(\sum_{i=1}^{r+s} \bigg(\frac{n\kappa_iT_{\mathfrak p}}{T_i}-\frac{d_i^2(1-\kappa_i)T_i}{nT_{\mathfrak p}}\bigg)-2n+\beta\sum_{i=1}^{r+s} \bigg(\frac{\kappa_i}{T_i}+\frac{d_i^2(1-\kappa_i)T_i}{nT_{\mathfrak p}(\beta+nT_{\mathfrak p})}\bigg)\bigg)
\\
&<\frac1{4\beta}\bigg(\sum_{i=1}^{r+s} \frac{n^2\kappa_iT_{\mathfrak p}^2-d_i^2(1-\kappa_i)T_i^2}{nT_{\mathfrak p}T_i}-2n+\beta\sum_{i=1}^{r+s}\frac{n^2\kappa_iT_{\mathfrak p}^2+d_i^2(1-\kappa_i)T_i^2}{n^2T_{\mathfrak p}^2T_i}\bigg).
\end{align*}
In view of~\eqref{hyp_thm_beta_zero+}, if~$\beta<\Lambda_0^+$, then
\begin{align*}
S(g) <\frac1{8\beta}\bigg(\sum_{i=1}^{r+s} \frac{n^2\kappa_iT_{\mathfrak p}^2-d_i^2(1-\kappa_i)T_i^2}{nT_{\mathfrak p}T_i}-2n\bigg)<0.
\end{align*}
In this case, again, estimate~\eqref{estimate+} holds.

Choose $p$ and $q$ such that
\begin{align}\label{nota_aqTp}
d_pT_p=\min_{i=1,\ldots,r+s}d_iT_i,\qquad \alpha_q=\min_{i=1,\ldots,r+s}\alpha_i.
\end{align}
Denote
\begin{align*}
\Gamma_0^+=\unm d_pT_p\min\Big\{1,\frac{\Lambda_0^+}{nT_{\mathfrak p}}\Big\}.
\end{align*}
By~\eqref{trace=1}, if $\alpha_q<\Gamma_0^+$, then
\begin{align*}
\beta=nT_{\mathfrak p}\bigg(\sum_{i=1}^{r+s}\frac{d_iT_i}{\alpha_i}-1\bigg)^{-1}<nT_{\mathfrak p}\Big(\frac{d_qT_q}{\alpha_q}-1\Big)^{-1}\le \frac{nT_{\mathfrak p}\alpha_q}{d_pT_p-\alpha_q}\le \frac{2nT_{\mathfrak p}\alpha_q}{d_pT_p}\le\Lambda_0^+,
\end{align*}
which means~\eqref{estimate+} holds.

Choose $l$ such that
\begin{align}\label{nota_al}
\alpha_l=\max_{i=1,\ldots,r+s}\alpha_i.
\end{align}
For $\epsilon>0$, denote
\begin{align*}
\Gamma_\infty^+(\epsilon)=\frac{2\Lambda_\infty^+(\epsilon)^2}{\min_{i=1,\ldots,r+s} d_i (1-\kappa_i)}\sum_{i=1}^{r} \frac{d_i\kappa_i}{\Gamma_0^+}.
\end{align*}
As we showed above, if $\beta>\Lambda_\infty^+(\epsilon)$ or $\alpha_q<\Gamma_0^+$, then~\eqref{estimate+} holds. Assuming $\alpha_q\ge\Gamma_0^+$ and $\alpha_l>\Gamma_\infty^+(\epsilon)$, we find
\begin{align*}
S(g) &<-\unc \sum_{i=1}^{r+s} \frac{\alpha_i}{\beta^2} d_i (1-\kappa_i)+ \unc \sum_{i=1}^{r} \frac{\kappa_i d_i}{\alpha_i} < -\frac{\alpha_l}{4\beta^2} d_l (1-\kappa_l) + \unc \sum_{i=1}^{r} \frac{\kappa_i d_i}{\alpha_q}
\\
&\le-\frac{\Gamma_\infty^+(\epsilon)}{4\Lambda_\infty^+(\epsilon)^2} \min_{i=1,\ldots,r+s}d_i (1-\kappa_i) + \unc \sum_{i=1}^{r} \frac{\kappa_i d_i}{\Gamma_0^+}<-\unc \sum_{i=1}^{r} \frac{\kappa_i d_i}{\Gamma_0^+}<0.
\end{align*}
Thus, the inequality $\alpha_l>\Gamma_\infty^+(\epsilon)$ implies~\eqref{estimate+}.

Let $\mathcal C_{\epsilon}^+$ be the set of metrics $g\in\mathcal M_T^+$ satisfying~\eqref{metric} with
\begin{align*}
\min\{\Lambda_0^+,\Gamma_0^+\}\le\min\{\beta,\alpha_1,\ldots,\alpha_{r+s}\}\le\max\{\beta,\alpha_1,\ldots,\alpha_{r+s}\}\le\max\{\Lambda_\infty^+(\epsilon),\Gamma_\infty^+(\epsilon)\}.
\end{align*}
Clearly, this set is compact. Summarising the arguments above, we conclude that~\eqref{estimate+} holds for all~$g\in\mathcal M_T^+\setminus\mathcal C_{\epsilon}^+$.
\end{proof}

With Lemma~\ref{lemmacompact+} at hand, we can prove Theorem~\ref{sufcon+} using the approach from~\cite[Proof of Theorem~3.3]{APZ20}. The main idea behind this approach goes back to~\cite{MGAP18}.

\begin{proof}[Proof of Theorem \ref{sufcon+}]
Denote $U= \tr_QT -d_mT_m$. For $t>U$, consider the metric $g_t\in\mca_{K}$ satisfying
\begin{align*}
g_t=  t Q|_{\pg} &+  t Q|_{\kg_1} + \cdots +  t Q|_ {\kg_{m-1}} 
\\ &+ \phi(t)Q|_{\kg_m} +  tQ|_ {\kg_{m+1}} + \cdots +   tQ|_ {\kg_{r+s}}, 
\qquad 
\phi(t) = \frac {d_mT_mt}{t- U}.
\end{align*}
Straightforward verification shows that 
$g_t$ lies in~$\mca_{T}^+$. 
By Corollary~\ref{scalar},
\begin{align*}
S(g_t) &=  \frac1{4t}d_m (1-\kappa_m) -  \frac{\phi(t)}{4t^2} d_m (1-\kappa_m) - \frac{1}{4t} \sum_{i=1}^{r+s} d_i (1-\kappa_i)
\\ &\hphantom{=}~- \frac{n}{2t} + \frac{1}{4t} \sum_{i=1}^{r} \kappa_i d_i - \frac{\kappa_md_m}{4t} + \frac{\kappa_md_m}{4\phi(t)} \\
&=  \frac1{4t}d_m (1-\kappa_m) -  \frac{\phi(t)}{4t^2} d_m (1-\kappa_m)- \frac{3n}{4t} +\frac{\dim\mathfrak k}{4t} - \frac{\kappa_md_m}{4t} + \frac{\kappa_md_m}{4\phi(t)}.
\end{align*}
Furthermore, in light of~\eqref{hyp_thm_beta_inf+},
\begin{align}\label{t2St}
4\lim_{t\to\infty}t^2\frac{d}{dt} S(g_t) & = -d_m(1-\kappa_m) - \dim \kg + 3 n +\kappa_md_m+ \frac{\kappa_mU}{ T_m}  \notag 
\\
& = -d_m(1-\kappa_m) - \dim \kg + 3 n + \frac{\kappa_m \tr_QT}{ T_m}  < 0.
\end{align}
We conclude that $\frac{d}{dt} S(g_t)<0$ for sufficiently large $t$, which implies the existence of $t_0\in(U,\infty)$ such that
\[
S(g_{t_0}) >\lim_{t \to \infty} S(g_t) = \frac{\kappa_m}{4T_m}.
\]
Using Lemma \ref{lemmacompact+} with 
\[
\epsilon =  \unm\Big(S(g_{t_0}) - \frac{\kappa_m}{4T_m}\Big)>0
\] yields
\begin{equation}\label{maxcompact}
S(h) < \frac{\kappa_m}{4T_m}+\epsilon=\unm S(g_{t_0})+\frac{\kappa_m}{8T_m}< S(g_{t_0}),\qquad h\in \mathcal{M}^{+}_{T}\setminus\cca_\epsilon^+.
\end{equation} 
Since $\cca_\epsilon^+$ is compact, the functional $S|_{\mathcal{C}^{+}_{\epsilon}}$ attains its global maximum at some ${g_{\mathrm{mx}} \in \cca_\epsilon^+}$. Obviously, $g_{t_0}$ lies in $\cca_\epsilon^+$. Therefore, by~\eqref{maxcompact}, $S(h) \leq S(g_{\mathrm{mx}})$ for all $h\in\mathcal{M}^{+}_{T}$.
\end{proof}

\subsection{Global maxima on $\mca_{T}^{-}$}\label{sec_-}

Now we focus on the space $\mca_{T}^{-}$.

\begin{theorem}\label{sufcon-}
Suppose $T$ is a left-invariant (0,2)-tensor field on $G$ satisfying~(\ref{T}) for some $T_{\pg},T_1,\ldots,T_{r+s}>0$. If condition~(\ref{hyp_thm_beta_zero+}) holds, then the functional $S |_{\mathcal{M}^{-}_{T}}$ attains its global maximum.
\end{theorem}

The proof relies on the following estimate.

\begin{lemma}\label{lemmacompact-}
Assume that~(\ref{hyp_thm_beta_zero+}) holds. Given $\theta>0$, there exists a compact set $\mathcal C_\theta^-\subset\mathcal{M}^-_{T}$ such that $S(g) < -\theta$ for every $g \in \mathcal{M}^-_{T} \setminus\cca_\theta^-$.
\end{lemma}

\begin{proof}
Let $g\in\mathcal M^+_T$ be a metric satisfying~\eqref{metric}. Then
\begin{align}\label{trace=-1}
\tr_gT=-\frac{nT_{\mathfrak p}}\beta+\sum_{i=1}^{r+s}\frac{d_iT_i}{\alpha_i}=-1,
\end{align}
which implies $\beta<nT_{\mathfrak p}$. Moreover,
\begin{align*}
\frac{\beta}{\alpha_j}<\frac\beta{d_jT_j}\sum_{i=1}^{r+s}\frac{d_iT_i}{\alpha_i}=\frac{nT_{\mathfrak p}-\beta}{d_jT_j}, \qquad j=1,\ldots,r+s.
\end{align*}
Given $\theta>0$, denote
\begin{align*}
\Lambda_0^-(\theta)=\frac1{4\theta}\bigg|\sum_{i=1}^{r+s} \frac{n^2\kappa_iT_{\mathfrak p}^2-d_i^2(1-\kappa_i)T_i^2}{nT_{\mathfrak p}T_i}-2n\bigg|.
\end{align*}
Corollary~\ref{scalar} implies
\begin{align*}
S(g)&<\frac1{4\beta}\bigg(\sum_{i=1}^{r+s} \bigg(\kappa_i\frac{nT_{\mathfrak p}-\beta}{T_i}-\frac{d_i^2(1-\kappa_i)T_i}{nT_{\mathfrak p}-\beta}\bigg)-2n \bigg)
\\
&<\frac1{4\beta}\bigg(\sum_{i=1}^{r+s} \bigg(\kappa_i\frac{nT_{\mathfrak p}}{T_i}-\frac{d_i^2(1-\kappa_i)T_i}{nT_{\mathfrak p}}\bigg)-2n \bigg)
\\
& =\frac1{4\beta}\bigg(\sum_{i=1}^{r+s} \frac{n^2\kappa_iT_{\mathfrak p}^2-d_i^2(1-\kappa_i)T_i^2}{nT_{\mathfrak p}T_i}-2n\bigg).
\end{align*}
By~\eqref{hyp_thm_beta_zero+}, if $\beta<\Lambda_0^-(\theta)$, then $S(g)<-\theta$.

Choose $p$ and $q$ as in~\eqref{nota_aqTp}. For $\theta>0$, denote
\begin{align*}
\Gamma_0^-(\theta)=\frac{d_pT_p\Lambda_0^-(\theta)}{nT_{\mathfrak p}}.
\end{align*}
If $\alpha_q<\Gamma_0^-(\theta)$, then
\begin{align*}
\beta=nT_{\mathfrak p}\bigg(\sum_{i=1}^{r+s}\frac{d_iT_i}{\alpha_i}+1\bigg)^{-1}<\frac{nT_{\mathfrak p}\alpha_q}{d_qT_q}\le \frac{nT_{\mathfrak p}\alpha_q}{d_pT_p}<\Lambda_0^-(\theta),
\end{align*}
which means $S(g)<-\theta$.

Choose $l$ as in~\eqref{nota_al}. For $\theta>0$, denote
\begin{align*}
\Gamma_\infty^-(\theta)=\frac{n^2T_{\mathfrak p}^2}{\min_{i=1,\ldots r+s} d_i (1-\kappa_i)}\bigg(4\theta+\sum_{i=1}^r\frac{\kappa_i d_i}{\Gamma_0^-(\theta)}\bigg).
\end{align*}
As shown above, if $\alpha_q<\Gamma_0^-(\theta)$, then $S(g)<-\theta$. Recalling that $\beta<nT_{\mathfrak p}$ and assuming that $\alpha_q\ge\Gamma_0^-(\theta)$ and $\alpha_l>\Gamma_\infty^-(\theta)$, we obtain
\begin{align*}
S(g) &< -\frac{\alpha_l}{4\beta^2} d_l(1-\kappa_l) + \unc \sum_{i=1}^{r} \frac{\kappa_i d_i}{\alpha_q}
\\
&\le-\frac{\Gamma_\infty^-(\theta)}{4n^2T_{\mathfrak p}^2} \min_{i=1,\ldots,r+s}d_i (1-\kappa_i) + \unc \sum_{i=1}^{r} \frac{\kappa_i d_i}{\Gamma_0^-(\theta)}<-\theta.
\end{align*}
Thus, the inequality $\alpha_l>\Gamma_\infty^-(\theta)$ implies $S(g)<-\theta$.

Let $\mathcal C_{\theta}^-$ be the set of those $g\in\mathcal M_T^+$ that satisfy~\eqref{metric} with
\begin{align*}
\min\{\Lambda_0^-(\theta),\Gamma_0^-(\theta)\}&\le\min\{\beta,\alpha_1,\ldots,\alpha_{r+s}\} \\ &\le\max\{\beta,\alpha_1,\ldots,\alpha_{r+s}\}\le\max\{nT_\pg,\Gamma_\infty^-(\theta)\}.
\end{align*}
This set is compact. By the arguments above, $S(g)<-\theta$ whenever $g$ lies in $\mathcal M_T^+\setminus\mathcal C_{\theta}^-$.
\end{proof}

\begin{proof}[Proof of Theorem \ref{sufcon-}]
Fix a metric $h\in\mathcal M_T^-$. Applying Lemma~\ref{lemmacompact-} with $\theta=|S(h)|+1$, we conclude that 
\begin{align*}
S(g) < -|S(h)|-1<S(h)
\end{align*}
for all $g\in\mathcal M_T^-$ outside a compact set~$\mathcal C_\theta^-\subset \mathcal M_T^-$. Clearly, $h$ lies in $\mathcal C_\theta^-$, and the functional $S|_{\mathcal C_\theta^-}$ attains its global maximum at some~$h_{\mathrm{mx}}\in\mathcal C_\theta^-$. This implies $S(g)\le S(h_{\mathrm{mx}})$ for all $g\in\mathcal M_T^-$.
\end{proof}

\subsection{Critical points on $\mca_{T}^{0}$}\label{sec_0}

If $r+s=1$ in formula~\eqref{dec_k}, then
\begin{align}\label{T1}
T =  - T_{\pg} Q|_{\pg} + T_1 Q|_{\kg_1} 
\end{align}
for some $T_\pg,T_1>0$. In this case, straightforward analysis shows that $S|_{\mathcal M_T^0}$ has no critical points unless
\begin{align}\label{cr_pt_M0_1}
d_1T_1^2 + 2nT_\pg (2T_1 - \kappa_1T_\pg)=0.
\end{align}
On the other hand, when~\eqref{cr_pt_M0_1} holds, the scalar curvature of every metric in $\mathcal M_T^0$ equals~0. If $r+s=2$, we are able to obtain a complete classification of the critical points of~$S|_{\mathcal M_T^0}$. We present this classification in Theorem~\ref{thm_M0} below. While its statement is quite bulky, its conditions are easy to verify once the tensor field $T$ and the geometric parameters of $G$ and $K$ are given. According to Table~3 in~\cite[Section~7.H]{Bss}, the sum $r+s$ can be greater than 2 only if $(\ggo,\kg)$ is one of the pairs
\begin{align*}
&(\sug(p,q),\sug(p)\oplus \sug(q) \oplus \RR), && 1<p \leq q, \\ &(\mathfrak{so}(4,m),\mathfrak{so}(4)\oplus \mathfrak{so}(m)), && m\ge4, \\ &(\mathfrak{so}(3,4),\mathfrak{so}(3)\oplus \mathfrak{so}(4)).
\end{align*}
It seems difficult to classify the critical points of~$S|_{\mathcal M_T^0}$ in these cases without using software, such as Maple, to solve the Euler--Lagrange equations numerically. Nevertheless, for all values of~$r+s$, the following result holds.

\begin{proposition}\label{prop_S=0}
If $g$ is a critical point of $S|_{\mathcal M_T^0}$, then $S(g)=0$.
\end{proposition}

\begin{proof}
According to Proposition~\ref{variational_lemma}, the fact that $g$ is a critical point of $S|_{\mathcal M_T^0}$ implies equality~\eqref{prescribed}. Taking the trace on both sides of this equality, we obtain
\begin{align*}
S(g)=\tr_g\Ricci (g)=c\tr_gT=0.
\end{align*}
\end{proof}

Assume that $r+s=2$ in formula~\eqref{dec_k}. For the list of $\kg$ satisfying this assumption, see Table~3 in~\cite[Section~7.H]{Bss}. Equality~\eqref{T} becomes
\begin{equation}\label{T2}
T = -T_{\pg} Q|_{\pg} + T_1 Q|_{\kg_1} + T_2 Q|_{\kg_2}.  
\end{equation}
It will be convenient for us to denote
\begin{align*}
a = nd_2 (1-\kappa_2) T_{\pg},\qquad 
b &= d_1^2 (1-\kappa_1)T_1 -  d_2^2 (1-\kappa_2) T_2 + 2n^2 T_\pg -  \frac{n^2\kappa_1 T_{\pg}^2}{T_1},\\
c &=-2 n d_2 T_2 + \frac{2nd_2\kappa_1 T_\pg T_2}{T_1} -  nd_2 \kappa_2 T_\pg, \\ d&=- \frac{d_2^2 \kappa_1T_2^2}{T_1} + \kappa_2 d_2^2 T_2.
\end{align*}
The proof of Theorem~\ref{thm_M0} below shows that the variational properties of $S|_{\mathcal M_T^0}$ are largely determined by those of the polynomial
\begin{align*}
P(x)=ax^3+bx^2+cx+d.
\end{align*}
The discriminant of this polynomial is
\begin{align*}
D=18abcd-4b^3d+b^2c^2-4ac^3-27a^2d^2.
\end{align*}
Denote
\begin{align*}
R_t=-\frac b{3a},\qquad R_d= \frac{9ad-bc}{2(b^2-3ac)},\qquad R_s=\frac{4 a b c -9 a^2 d -b^3}{a (b^2 - 3 a c)}.
\end{align*}
According to the classical theory of cubic equations (see, e.g.,~\cite{Jan10} for a modern interpretation), if $D=0$ and $b^2=3ac$, then $x=R_t$ is a triple root of~$P(x)$. It is also a saddle point. If $D=0$ and $b^2\ne3ac$, then $x=R_d$ and $x=R_s$ are a double root and a simple root of $P(x)$, respectively. Both are local extremum points.

\begin{theorem}\label{thm_M0}
Assume that $r+s=2$ in formula~(\ref{dec_k}). The scalar curvature functional $S|_{\mca_T^0}$ does not have a global minimum. Critical points of other types exist under the following conditions:
\begin{enumerate}
\item
A saddle if and only if
\begin{align}\label{cond_saddle}
D=0,\qquad b^2=3ac,\qquad \frac{d_2 T_2}{n T_\pg}<R_t.
\end{align}

\item
A global maximum if and only if
\begin{align}\label{cond_gmx}
D=0,\qquad b^2\ne3ac,\qquad R_s\le\frac{d_2 T_2}{n T_\pg}<R_d.
\end{align}

\item
A local maximum that is not a global maximum if and only if
\begin{align}\label{cond_lmx}
D=0,\qquad b^2\ne3ac,\qquad \frac{d_2 T_2}{n T_\pg}<R_s<R_d.
\end{align}

\item
A local minimum if and only if
\begin{align}\label{cond_lmn}
D=0,\qquad b^2\ne3ac,\qquad \frac{d_2 T_2}{n T_\pg}<R_d<R_s.
\end{align}
\end{enumerate}
When it exists, the critical point of $S|_{\mca_T^0}$ is unique up to scaling.
\end{theorem}

Let us make a few remarks in preparation for the proof. Consider a metric $g\in\mathcal M_T^0$. There are $\beta,\alpha_1,\alpha_2>0$ such that
\begin{equation}\label{g2}
g = \beta Q|_{\pg} + \alpha_1 Q|_{\kg_1} + \alpha_2 Q|_{\kg_2}.
\end{equation}
The equality $\tr_gT=0$ implies
\begin{align*}
-\frac{nT_{\pg}}{\beta}+\frac{d_1T_1}{\alpha_1}+\frac{d_2T_2}{\alpha_2}=0,\qquad \frac\beta{\alpha_1}=\frac1{d_1T_1}\bigg(nT_{\pg}-\frac{d_2T_2\beta}{\alpha_2}\bigg), \qquad \frac{\alpha_2}\beta>\frac{d_2T_2}{nT_{\pg}}.
\end{align*}
By Corollary~\ref{scalar},
\begin{align}\label{eq_sc_2id}
S(g) &= \frac1{4\beta}\Big(-\frac{\alpha_1}{\beta} d_1 (1-\kappa_1)-\frac{\alpha_2}{\beta} d_2 (1-\kappa_2) -  2n + \frac{\kappa_1 d_1\beta}{\alpha_1}+ \frac{\kappa_2 d_2\beta}{\alpha_2}\Big)\notag
\\
&=\frac\beta{4\alpha_2(d_2T_2\beta-nT_{\pg}\alpha_2)}P\Big(\frac{\alpha_2}\beta\Big).
\end{align}
Note that the factor in front of $P(\frac{\alpha_2}\beta)$ is necessarily negative. Our arguments will involve two curves, $\gamma_1$ and $\gamma_2$, in the space $\mathcal M_T^0$ given by the formulas
\begin{align*}
\gamma_1(t)&= \beta Q|_{\pg} - \frac{d_1T_1e^t\alpha_2\beta}{d_2T_2\beta-nT_{\pg}e^t\alpha_2} Q|_{\kg_1} + e^t\alpha_2 Q|_{\kg_2}, &&t>\ln\frac{d_2T_2\beta}{nT_{\pg}\alpha_2},
\\
\gamma_2(t)&=e^t\beta Q|_{\pg} + e^t\alpha_1 Q|_{\kg_1} + e^t\alpha_2 Q|_{\kg_2}, &&t\in\mathbb R.
\end{align*}
Both these curves pass through $g$ at $t=0$. 

\begin{lemma}\label{lem_crit=root}
The metric $g$ given by~\eqref{g2} is a critical point of $S|_{\mathcal M_T^0}$ if and only if $x=\frac{\alpha_2}\beta$ is a multiple root of~$P(x)$.
\end{lemma}

\begin{proof}
Assume $g$ is a critical point of $ S|_{\mathcal M_T^0}$. Proposition~\ref{prop_S=0} implies $S(g)=0$. In light of~\eqref{eq_sc_2id}, this means $x=\frac{\alpha_2}\beta$ must be a root of~$P(x)$. Furthermore, because $g$ is a critical point of $S|_{\mca_T^0}$,
\begin{align*}
0&=\frac{d}{dt}S(\gamma_1(t))|_{t=0}=\frac d{dt}\Big(\frac\beta{4e^t\alpha_2(d_2T_2\beta-nT_{\pg}e^t\alpha_2)}P\Big(\frac{e^t\alpha_2}\beta\Big)\Big)\Big|_{t=0}
\\ 
&=\frac d{dt}\frac\beta{4e^t\alpha_2(d_2T_2\beta-nT_{\pg}e^t\alpha_2)}\Big|_{t=0}P\Big(\frac{\alpha_2}\beta\Big) + \frac\beta{4 \alpha_2(d_2T_2\beta-nT_{\pg}\alpha_2)} \frac d{dt}P\Big(\frac{e^t\alpha_2}\beta\Big)\Big|_{t=0}\\
& = \frac1{4 (d_2T_2\beta-nT_{\pg}\alpha_2)}\frac{d}{dx}P(x)|_{x=\frac{\alpha_2}\beta}.
\end{align*}
Thus, the derivative of $P(x)$ at $x=\frac{\alpha_2}\beta$ vanishes. This proves the ``only if" part of the claim.

Assume that $x=\frac{\alpha_2}\beta$ is a multiple root of $P(x)$. We need to show that $g$ is a critical point of~$S|_{\mathcal M_T^0}$. Clearly, the vectors tangent to the curves $\gamma_1$ and $\gamma_2$ at $g$ are linearly independent. Therefore, it suffices to prove that
\begin{align*}
\frac{d}{dt}S(\gamma_1(t))|_{t=0}=\frac{d}{dt}S(\gamma_2(t))|_{t=0}=0.
\end{align*}
Computing as above, we find
\begin{align*}
\frac{d}{dt}&S(\gamma_1(t))|_{t=0} \\ &=\frac d{dt}\frac\beta{4e^t\alpha_2(d_2T_2\beta-nT_{\pg}e^t\alpha_2)}\Big|_{t=0}P\Big(\frac{\alpha_2}\beta\Big) + \frac1{4(d_2T_2\beta-nT_{\pg}\alpha_2)}\frac{d}{dx}P(x)|_{x=\frac{\alpha_2}\beta} =0.
\end{align*}
Formula~\eqref{eq_sc_2id} implies
\begin{align*}
\frac{d}{dt}S(\gamma_2(t))|_{t=0}=\frac{d}{dt}S(e^tg)|_{t=0}=\frac{d}{dt}e^{-t}S(g)|_{t=0}=-\frac{\beta}{4\alpha_2(d_2T_2\beta-nT_{\pg}\alpha_2)}P\Big(\frac{\alpha_2}\beta\Big)=0.
\end{align*}
\end{proof}

\begin{proof}[Proof of Theorem~\ref{thm_M0}]
Recalling that $P(x)$ is a cubic polynomial, we find
\begin{align*}
\lim_{t\to\infty}S(\gamma_1(t))&=\lim_{t\to\infty}\frac\beta{4e^t\alpha_2(d_2T_2\beta-nT_{\pg}e^t\alpha_2)}P\Big(\frac{e^t\alpha_2}\beta\Big)=-\infty.
\end{align*}
Consequently, $S|_{\mathcal M_T^0}$ never attains its global minimum. This proves the first statement.

Suppose $g$ is a saddle point of $S|_{\mathcal M_T^0}$. Proposition~\ref{prop_S=0} implies $S(g)=0$. Moreover, every neighbourhood of $g$ in $\mathcal M_T^0$ contains a metric with negative scalar curvature and one with positive scalar curvature. By Lemma~\ref{lem_crit=root}, $x=\frac{\alpha_2}\beta$ is a multiple root of~$P(x)$. Formula~\eqref{eq_sc_2id} shows that every interval around $x=\frac{\alpha_2}\beta$ contains a point where $P(x)$ is positive and one where $P(x)$ is negative. This is only possible if $x=\frac{\alpha_2}\beta$ is a triple root. By the classical theory of cubic equations, conditions~\eqref{cond_saddle} hold. Conversely, these conditions ensure that $P(x)$ has a triple root at $x=R_t$. Consider a metric $g_{\mathrm{sdl}}\in\mathcal M_T^0$ defined by
\begin{equation*}
g_{\mathrm{sdl}} = Q|_{\pg}-\frac{d_1T_1R_t}{d_2T_2-nT_{\pg}R_t} Q|_{\kg_1} + R_t Q|_{\kg_2}.
\end{equation*}
Lemma~\ref{lem_crit=root} implies that $g_{\mathrm{sdl}}$ is a critical point of~$S|_{\mathcal M_T^0}$. Using~\eqref{eq_sc_2id}, one easily shows that every neighbourhood of $g_{\mathrm{sdl}}$ contains a metric with negative scalar curvature and one with positive scalar curvature. In light of Proposition~\ref{prop_S=0}, this means $g_{\mathrm{sdl}}$ is a saddle point.

The functional $S|_{\mathcal M_T^0}$ attains its global maximum if and only if $S(g_{\mathrm{gmx}})=0$ for some $g_{\mathrm{gmx}}\in\mathcal M_T^0$ and $S(h)\le0$ for all $h\in\mathcal M_T^0$. Formula~\eqref{eq_sc_2id} implies that this happens if and only if $P(x)$ has a double root in the interval $\big(\frac{d_2 T_2}{n T_\pg},\infty\big)$ and is nonnegative on this interval. Conditions~\eqref{cond_gmx} are necessary and sufficient for $P(x)$ to have such properties.

Next, $S|_{\mathcal M_T^0}$ has a local maximum that is not a global maximum if and only if $S(g_{\mathrm{lmx}})=0$ for some $g_{\mathrm{lmx}}\in\mathcal M_T^0$, $S(h)\le0$ for all $h$ in a neighbourhood of $g_{\mathrm{lmx}}$, and the scalar curvature of at least one metric in $\mathcal M_T^0$ is positive. This is equivalent to $P(x)$ having a simple root in the interval $\big(\frac{d_2 T_2}{n T_\pg},\infty\big)$ and a double root in $(R_s,\infty)$. Conditions~\eqref{cond_lmx} are necessary and sufficient for $P(x)$ to have such properties.

Analogously, $S|_{\mathcal M_T^0}$ has a local minimum if and only if $P(x)$ has a double root in $\big(\frac{d_2 T_2}{n T_\pg},\infty\big)$ and is nonpositive in a neighbourhood of this root. Conditions~\eqref{cond_lmn} are necessary and sufficient for this.

Finally, in view of Lemma~\ref{lem_crit=root}, $S|_{\mathcal M_T^0}$ can have at most one critical point up to scaling since a cubic polynomial can have at most one multiple root.
\end{proof}

\subsection{Summary}\label{sec_sum}

The results of Sections~\ref{sec_variational}--\ref{sec_0} enable us to make several conclusions about the solvability of~\eqref{prescribed}. We summarise these conclusions in Theorem~\ref{thm_summary} below. The constant $c$ in~\eqref{prescribed} must be positive if $T$ satisfies~\eqref{T}. This is an immediate consequence of the formulas for the Ricci curvature obtained in Section~\ref{secRicci}. 

\begin{theorem}\label{thm_summary}
Suppose $T$ is a left-invariant (0,2)-tensor field on $G$ given by~(\ref{T}).
\begin{enumerate}
\item
If~(\ref{hyp_thm_beta_zero+}) holds, then there exists at least one pair $(g,c)\in\mathcal M_K\times (0,\infty)$ satisfying~(\ref{prescribed}).

\item
If both~(\ref{hyp_thm_beta_inf+}) and~(\ref{hyp_thm_beta_zero+}) hold, then there are at least two pairs $(g,c)\in\mathcal M_K\times (0,\infty)$ that satisfy~(\ref{prescribed}) and have non-homothetic metrics~$g$.

\item
If $r+s=2$ and conditions~(\ref{cond_saddle}), (\ref{cond_gmx}), (\ref{cond_lmx}) or~(\ref{cond_lmn}) hold, then there exists at least one pair $(g,c)\in\mathcal M_K\times (0,\infty)$ satisfying~(\ref{prescribed}).

\end{enumerate}
\end{theorem}

\begin{proof}
Statements~1 and~3 follow from Theorems~\ref{sufcon-} and~\ref{thm_M0} combined with Proposition~\ref{variational_lemma}. Next, assume that~(\ref{hyp_thm_beta_inf+}) and~(\ref{hyp_thm_beta_zero+}) hold. According to Theorems~\ref{sufcon+} and~\ref{sufcon-}, the functionals $S|_{\mathcal M_T^+}$ and $S|_{\mathcal M_T^-}$ attain their global maxima at some $g_1\in\mathcal M_T^+$ and $g_2\in\mathcal M_T^-$. Proposition~\ref{variational_lemma} implies that both $g_1$ and $g_2$ have Ricci curvature equal to $T$ up to scaling. These metrics cannot be homothetic because $\tr_{g_1}T$ and $\tr_{g_2}T$ are not of the same sign.
\end{proof}

When $r+s=1$, Theorem~\ref{thm_summary} is essentially optimal. We explain this in detail in Remark~\ref{rem_optimal}. At the same time, when $r+s\ge2$, it seems that~(\ref{hyp_thm_beta_inf+}) and~\eqref{hyp_thm_beta_zero+} may fail to hold even if $S|_{\mca_T^+}$ and $S|_{\mca_T^-}$ attain their global maxima. Indeed, on compact Lie groups, inequalities that are similar in spirit to these provide merely a ``linear approximation" to the necessary and sufficient conditions for the existence of a critical point; see~\cite[Section~5]{APZ20}.

Different pairs $(g,c)\in\mathcal M_K\times(0,\infty)$ satisfying~\eqref{prescribed} must have distinct~$c$. More precisely, by Theorem~\ref{thm_no_c}, if
\begin{align*}
\Ricci (g_1)=\Ricci(g_2)=cT,\qquad g_1,g_2\in\mathcal M_K,
\end{align*}
then $g_1$ and $g_2$ are equal up to scaling.

\begin{remark}\label{rem_r+s=1}
The discussion at the beginning of Section~\ref{sec_0} shows that~\eqref{cr_pt_M0_1} is a sufficient condition for the solvability of~\eqref{prescribed} if $r+s=1$.

\end{remark}

\section{The case where $K$ is simple}\label{sec_simple}

As above, let $T$ be a left-invariant (0,2)-tensor field on~$G$. Assume that $K$ is simple. Our next result settles the question of solvability of~\eqref{prescribed} under this assumption. We do not use the variational approach developed in Section~\ref{sec_variational}; however, see Remarks~\ref{rem_optimal} and~\ref{rem_sim_max} below.

Since $K$ is simple, the numbers $r$ and $s$ in~\eqref{dec_k} equal~1 and~0, respectively. By Theorem~\ref{Ric}, if~\eqref{T1} holds for some $T_\pg,T_1>0$, the constant $c$ in~\eqref{prescribed} must be positive.

\begin{proposition}\label{prop_simple_gr}
Assume $K$ is simple. Let the tensor field $T$ satisfy~(\ref{T1}) for some $T_{\pg},T_1>0$.
\begin{enumerate}
\item
If
\begin{align}\label{sim_con1}
2nT_{\pg}(2  T_1-  \kappa_1T_{\pg}) < -d_1T_1^2
\end{align}
then there exists no metric $g\in\mathcal M_K$ such that~(\ref{prescribed}) holds.

\item
If
\begin{align*}
2n T_{\pg}(2  T_1-  \kappa_1T_{\pg}) = -d_1T_1^2\qquad
or\qquad 
2T_1- \kappa_1T_{\pg}  \ge 0,
\end{align*}
then there exists precisely one pair $(g,c)\in\mathcal M_K\times(0,\infty)$, up to scaling of $g$, such that~(\ref{prescribed}) holds.

\item
If
\begin{align*}
-d_1T_1^2 < 2n T_{\pg}(2 T_1- \kappa_1T_{\pg}) <0,
\end{align*}
then there are precisely two pairs $(g,c)\in\mathcal M_K\times(0,\infty)$, up to scaling of $g$, such that~(\ref{prescribed}) holds.
\end{enumerate}
\end{proposition}

\begin{proof}
Choose a metric $g\in\mathcal M_K$. There exist $\beta,\alpha_1>0$ such that
\begin{equation*}
g = \beta Q|_{\pg} + \alpha_1 Q|_{\kg_1}.
\end{equation*}
Theorem~\ref{Ric} and formula~\eqref{sum_dkappan} imply that $g$ satisfies~\eqref{prescribed} if and only if
\begin{align*}
\unc(\kappa_1(1 - x^2) + x^2) &= c T_1, \\
\unc(2+x) &= c T_{\pg},
\end{align*}
where $x=\frac{\alpha_1}{\beta}$. Clearing $c$ from the second line and substituting into the first, we obtain
\begin{align*}
(1-\kappa_1) T_{\pg} x^2 - T_1 x + \kappa_1 T_{\pg} - 2T_1 =0. 
\end{align*}
This is a quadratic equation with discriminant
\begin{align*}
E= T_1^2 + 4  (1 - \kappa_1) T_{\pg} (2T_1- \kappa_1 T_{\pg}).
\end{align*}
It has no solutions if $E<0$, precisely one positive solution if $E=0$ or $E\ge T_1^2$, and precisely two positive solutions if $0<E<T_1^2$. Together with~\eqref{sum_dkappan}, this implies the result.
\end{proof}

\begin{remark}\label{rem_optimal}
Proposition~\ref{prop_simple_gr} shows that Theorem~\ref{thm_summary} is essentially optimal in our current setting. Indeed, since $K$ is simple,~\eqref{hyp_thm_beta_zero+} becomes
\begin{align*}
\frac{\kappa_1n^2T_{\mathfrak p}^2- (1-\kappa_1)d_1^2T_1^2}{nT_{\mathfrak p}T_1}-2n<0.
\end{align*}
In view of~\eqref{sum_dkappan}, this is equivalent to
\begin{align*}
2nT_{\mathfrak p}(2T_1-\kappa_1T_{\mathfrak p})>- d_1T_1^2.
\end{align*}
Theorem~\ref{thm_summary} and Remark~\ref{rem_r+s=1} assert that~\eqref{hyp_thm_beta_zero+} and~\eqref{cr_pt_M0_1} are sufficient conditions for the solvability of~\eqref{prescribed}. Conversely, as Proposition~\ref{prop_simple_gr} shows, the existence of a pair $(g,c)\in\mathcal M_K\times(0,\infty)$ satisfying~\eqref{prescribed} implies that either~\eqref{hyp_thm_beta_zero+} or~\eqref{cr_pt_M0_1} must hold. Theorem~\ref{thm_summary} provides lower bounds on the number of solutions to~\eqref{prescribed}. Using Proposition~\ref{prop_simple_gr}, one can easily demonstrate that these bounds are sharp.
\end{remark}

\begin{remark}\label{rem_sim_max}
In our current setting, every metric satisfying~\eqref{prescribed} is, up to scaling, a global maximum point of $S|_{\mca_T^+}$, $S|_{\mca_T^-}$ or $S|_{\mca_T^0}$. This observation follows from the results of Section~\ref{sec_Ricci=cT} and Proposition~\ref{prop_simple_gr}.
\end{remark}

\section{Examples}\label{sec_examples}

Let us illustrate how the results of Sections~\ref{sec_Ricci=cT}--\ref{sec_simple} apply to specific groups.

\begin{example}\label{exa_simple}
Assume $G=\G_2^{\mathbb C}$ and $K=\G_2$. Then
\begin{align*}
r=1,\qquad s=0,\qquad n=d_1=14.
\end{align*}
Formula~\eqref{sum_dkappan} yields $\kappa_1=\unm$. Suppose $T$ is given by~\eqref{T1} with~$T_\pg,T_1>0$. Since $K$ is simple, Proposition~\ref{prop_simple_gr} applies. Formula~\eqref{sim_con1} becomes
\begin{align*}
4T_{\pg}T_1-T_{\pg}^2<-T_1^2.
\end{align*}
Equivalently,
\begin{align*}
\frac{T_1}{T_\pg}<(\sqrt5-2).
\end{align*}
If this holds,
then~\eqref{prescribed} has no solutions. Similarly, if
\begin{align*}
\frac{T_1}{T_\pg}=(\sqrt5-2)\qquad
\mbox{or}\qquad 
\frac{T_1}{T_\pg}\ge\unc,
\end{align*}
then there is one pair $(g,c)\in\mathcal M_K\times(0,\infty)$, up to scaling of $g$, that satisfies~(\ref{prescribed}). If
\begin{align*}
(\sqrt5-2)<\frac{T_1}{T_\pg}<\unc,
\end{align*}
there are two such pairs.
\end{example}

\begin{example}\label{ex_2ideals}
Assume $G=\SO^+(2,q)$ and $K=\SO(2)\times\SO(q)$ with $q\ge5$. Then
\begin{align*}
r=1,\qquad s=1,\qquad n=2q,\qquad d_1=\frac{q(q-1)}2,\qquad d_2=1.
\end{align*}
Using~\eqref{sum_dkappan}, we find
\begin{align*}
\kappa_1=\frac{q-2}q,\qquad\kappa_2=0.
\end{align*}
Suppose $T$ is given by~\eqref{T2} with~$T_\pg,T_1,T_2>0$. Inequality~\eqref{hyp_thm_beta_zero+} becomes
\begin{align}\label{exam_2id}
8q(q-2)T_{\mathfrak p}^2- q(q-1)^2T_1^2-2T_1T_2-16q^2T_\pg T_1<0.
\end{align}
According to Theorem~\ref{sufcon-}, if~\eqref{exam_2id} holds, then $S|_{\mca_T^-}$ attains its global maximum. In this case, there exists a metric~$g\in\mca_T^-$ with Ricci curvature~$cT$ for some~$c>0$. Inequality~\eqref{hyp_thm_beta_inf+} takes the form
\begin{align}\label{exam2plus}
-2q(q-2)T_\pg+q(4q+1)T_1+(q-2)T_2<0.
\end{align}
The set of triples~$(T_\pg,T_1,T_2)$ for which both~\eqref{exam_2id} and~\eqref{exam2plus} are satisfied is non-empty and open in~$\mathbb R^3$. We depict it in Figure~1 for $q=10$. We also indicate where~\eqref{exam_2id} holds without~\eqref{exam2plus}. Because these inequalities are invariant under scaling of~$(T_\pg,T_1,T_2)$, we make our sketch assuming~$T_\pg=1$. By Theorem~\ref{sufcon+}, if both~\eqref{exam_2id} and~\eqref{exam2plus} are satisfied, then $S|_{\mca_T^+}$ attains its global maximum. In this case, there exists a metric $g\in\mca_T^+$ with Ricci curvature $cT$ for some~$c>0$.

\begin{figure}[h]
\includegraphics[width=1\textwidth]{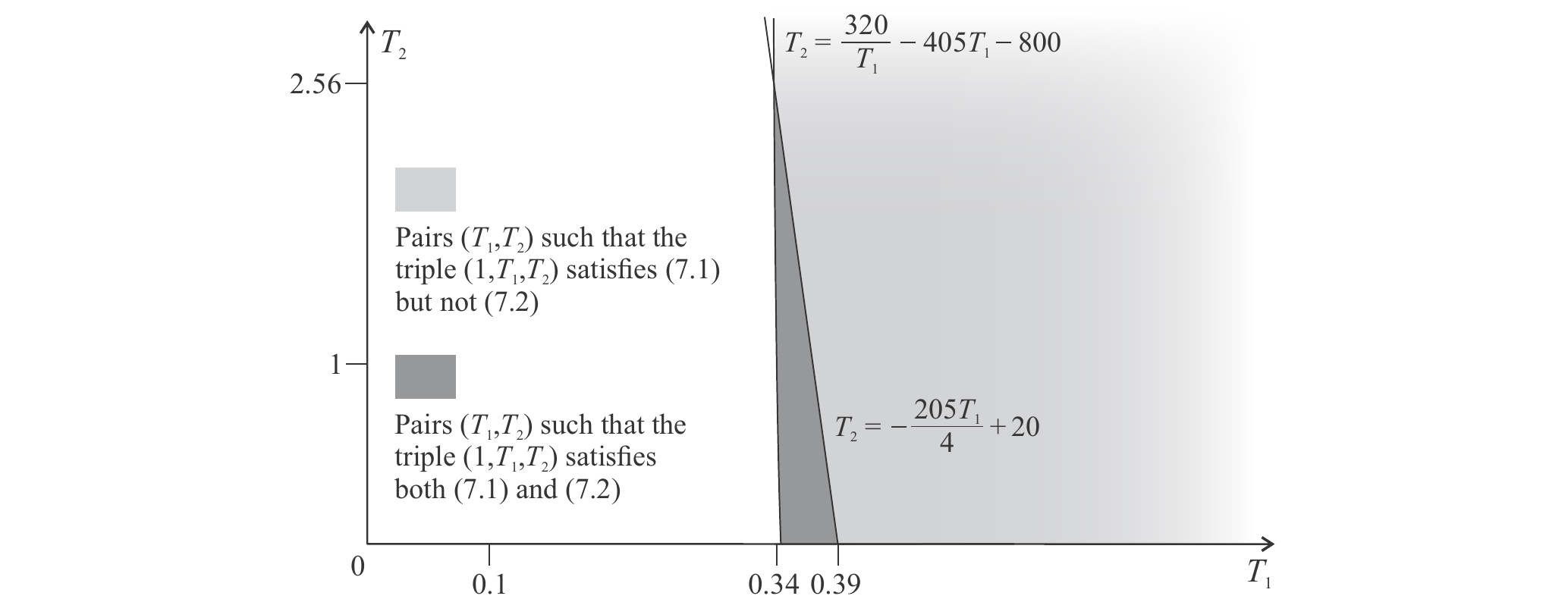}
\caption{\small{Existence of global maxima of $S|_{\mca_T^-}$ and $S|_{\mca_T^+}$ on $\SO^+(2,10)$}}
\end{figure}

Theorem~\ref{thm_M0} enables us to classify the critical points of~$S|_{\mca_T^0}$. For instance, suppose $q=5$ and $(T_\pg,T_1,T_2)=(1,1,1)$. Then the discriminant $D$ of the polynomial $P(x)$ satisfies
\begin{align*}
D=\tfrac{79948008}5\ne0. 
\end{align*}
By Theorem~\ref{thm_M0}, $S|_{\mca_T^0}$ has no critical points. To give another example, suppose $q=5$ and $(T_\pg,T_1,T_2)=\big(1,\tfrac15,\eta\big)$, where $\eta$ is the unique positive root of the polynomial
\begin{align*}
R(x)= -12x^3-4412x^2-583104x+4198144.
\end{align*}
Then $D=\eta^2R(\eta)=0$. Moreover, using an approximate value of $\eta$ calculated in Maple, we find $b^2\ne3ac$ and
\begin{align*}
R_s=\frac{\eta^3+1376\eta^2-121808\eta+778688}{10\eta^2-10160\eta+84640}<\frac{\eta}{10}<\frac{115\eta(16-\eta)}{\eta^2-1016\eta+8464}=R_d.
\end{align*}
This means $S|_{\mca_T^0}$ attains its global maximum.

One can use Maple to produce the graphs of $S|_{\mca_T^-}$, $S|_{\mca_T^+}$ and~$S|_{\mca_T^0}$; cf.~\cite[Section~5]{APZ20}.
\end{example}

\begin{example}
Assume $G=\SU(p,q)$ and $K=\SU(p)\times\U(q)$ with $2\le p\le q$. Then
\begin{align*}
r=2,\qquad s=1,\qquad n=2pq,\qquad d_1=p^2-1,\qquad d_2=q^2-1,\qquad d_3=1.
\end{align*}
As we showed in Example~\ref{exa_kappa},
\begin{align*}
\kappa_1=\frac p{p+q},\qquad\kappa_2=\frac q{p+q},\qquad\kappa_3=0.
\end{align*}
Suppose 
\begin{align*}
T = - T_{\pg} Q|_{\pg} + T_1 Q|_{\kg_1}+ T_1 Q|_{\kg_3}+ T_3 Q|_{\kg_3} 
\end{align*}
for some~$T_\pg,T_1,T_2,T_3>0$. Inequality~\eqref{hyp_thm_beta_zero+} becomes
\begin{align}\label{exam_3id}
%
%
4p^3q^2T_{\mathfrak p}^2T_2&-(p^2-1)^2qT_1^2T_2+4p^2q^3T_{\mathfrak p}^2T_1 \notag
\\
&-(q^2-1)^2pT_1T_2^2-(p+q)T_1T_2T_3-8p^2q^2(p+q)T_\pg T_1T_2<0.
\end{align}
By Theorem~\ref{sufcon-}, if~\eqref{exam_3id} holds, then $S|_{\mca_T^-}$ attains its global maximum. This is the case, e.g., when 
\begin{align*}
(T_\pg,T_1,T_2,T_3)=(1,1,1,4p^2q^2).
\end{align*}
Assume for simplicity that $T_1\ge T_2$. Then~\eqref{hyp_thm_beta_inf+} takes the form
\begin{align}\label{exam_3max}
-2pq^2T_\pg+q(p^2-1)T_1+qT_3 &<(p^3-5p^2q-4pq^2-2p)T_2.
\end{align}
By Theorem~\ref{sufcon-}, if both~\eqref{exam_3id} and this inequality hold, then $S|_{\mca_T^+}$ attains its global maximum. This happens, e.g., for $(p,q)=(2,12)$ and $(T_\pg,T_1,T_2,T_3)=\big(1,1,\tfrac38,1\big)$.
\end{example}

\section*{Acknowledgements}

The authors are grateful to Jorge Lauret and Cynthia Will for their careful reading of the paper and useful comments.

\end{document}